\RequirePackage{fix-cm}
\documentclass{amsart}
\usepackage{graphicx}
\usepackage{mathptmx}      
\usepackage{pst-all}      
\usepackage{url}
\usepackage{float}
\usepackage{rotating}
\usepackage{lscape}
\newtheorem{theorem}{Theorem}
\begin{document}

\title{Generating maps on surfaces}

\author{Thom Sulanke}
\address{Department of Physics, Indiana University, Bloomington, Indiana 47405}
   \email{tsulanke@indiana.edu}

\begin{abstract}
\normalsize

We describe procedures for generating all $2$-cell embedded simple graphs
with up to a fixed number of vertices on a given surface.
We also modify these procedures to generate closed $2$-cell embeddings and
polyhedral embeddings.
We give results of computer implementations of these procedures for seven
surfaces:
the sphere, the torus, the double torus, the projective plane, the Klein
bottle, the triple cross surface, and the quadruple cross surface.

\end{abstract}
\maketitle

\section{Introduction}
\label{introduction}

A {\em map} is a simple graph embedded in a surface such that every face is 
simply connected.
A {\em triangulation} is a map in which every face has three edges.
Section~\ref{definitions} contains more detailed definitions.
In Sections~\ref{recursive}, \ref{addedges}, and \ref{contractedges}
we describe the operations or local deformations which we apply to 
triangulations and maps.

The generation of triangulations and maps on surfaces such as the projective 
plane, the torus, and the Klein bottle has similarities
to the generation of triangulations and maps on the sphere.
There are also interesting differences.

Brinkmann and McKay provide procedures for generating triangulations and 
\linebreak
maps on the sphere \cite{MR2357364,plantripaper}.
They implement these
procedures in a computer program {\em plantri} \cite{plantri}.
We extend these techniques to non-spherical surfaces.

We know from the work of Steinitz \cite{StRa} that we can generate all the 
triangulations of the sphere with $n > 4$ vertices by applying 
the vertex splitting operation (Figure~\ref{split} 
and Section~\ref{contractedges}) to the 
triangulations of the sphere with $n-1$ vertices.
The single initial triangulation for this recursive process is $K_4$ 
embedded in the sphere, the (boundary of the) tetrahedron.
The inverse of the vertex splitting operation is the operation of 
edge contraction.
To assure that the repeated application of the vertex splitting operation 
generates
all triangulations it is required that the edge contraction operation can 
always be performed on any triangulation other than $K_4$
and that the result of this edge contraction operation is
also a triangulation.
For a triangulation of the sphere other than $K_4$ it is indeed 
always possible 
to find an edge for which the edge contraction operation can be applied
\cite{StRa}.

\begin{figure}
\centering
\psset{unit=.009\textwidth}
\parbox{.39\textwidth}{%
\centering
\ttfamily
\begin{pspicture}(-5,-5)(35,20)
\rput{180}(5,10){\pspolygon[linewidth=.3pt](.5,0)(5,.8)(5,-.8)}
\rput{200}(5,10){\psline[linewidth=.3pt](.5,0)(5,0)}
\psline[linewidth=.3pt]{*-*}(5,10)(15,10)
\psline[linewidth=.3pt]{*-*}(15,10)(25,10)
\rput{0}(25,10){\pspolygon[linewidth=.3pt](.5,0)(5,.8)(5,-.8)}
\rput{-20}(25,10){\psline[linewidth=.3pt](.5,0)(5,0)}
\rput{90}(15,10){\pspolygon[linewidth=.3pt](.5,0)(5,.8)(5,-.8)}
\rput{270}(15,10){\pspolygon[linewidth=.3pt](.5,0)(5,.8)(5,-.8)}
\end{pspicture}}
\parbox{.19\textwidth}{%
\centering
\ttfamily
\begin{pspicture}(0,0)(15,20)
\psline[linewidth=1pt]{->}(5,10)(10,10)
\rput(7.5,12){{E}}
\end{pspicture}}
\parbox{.39\textwidth}{%
\centering
\ttfamily
\begin{pspicture}(-5,-5)(35,20)
\rput{180}(5,10){\pspolygon[linewidth=.3pt](.5,0)(5,.8)(5,-.8)}
\rput{200}(5,10){\psline[linewidth=.3pt](.5,0)(5,0)}
\psline[linewidth=.3pt]{*-*}(5,10)(15,16)
\psline[linewidth=.3pt]{*-*}(15,16)(25,10)
\psline[linewidth=.3pt]{*-*}(25,10)(15,4)
\psline[linewidth=.3pt]{*-*}(15,4)(5,10)
\psline[linewidth=.3pt]{*-*}(15,4)(15,16)
\rput{0}(25,10){\pspolygon[linewidth=.3pt](.5,0)(5,.8)(5,-.8)}
\rput{-20}(25,10){\psline[linewidth=.3pt](.5,0)(5,0)}
\rput{90}(15,16){\pspolygon[linewidth=.3pt](.5,0)(5,.8)(5,-.8)}
\rput{270}(15,4){\pspolygon[linewidth=.3pt](.5,0)(5,.8)(5,-.8)}
\end{pspicture}}

\caption{Expansion for splitting a vertex}
\label{split}
\end{figure}
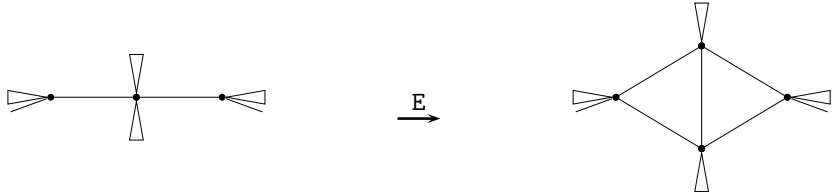

Triangulations of non-spherical surfaces can be generated in the same way.
However, for any surface other than the sphere 
there are many but a finite number of triangulations which do not have
any contractible edges \cite{MR1021367}.
An edge is not contractible if an attempt to apply the edge contraction 
operation would produce multiple edges.
We define a triangulation with no contractible edges to be an 
{\em irreducible triangulation}.
The initial triangulations used for 
generating all triangulations of a surface are the 
irreducible triangulations of this surface.
In Section~\ref{findirrtri} we provide a method for producing 
the class of irreducible triangulations of a surface.

We next turn to the generation of maps which might not be triangulations.
Maps on the sphere can be obtained from the triangulations of the sphere by
the operation of edge removal (Figure~\ref{removecorner} 
and Section~\ref{addedges}).
For the repeated application of the operation of edge removal to generate 
all the maps on the sphere, 
it must always be possible to perform the inverse operation of adding 
an edge to a map which is not a
triangulation and, in the process, obtain another map.
The edge must be added in such a way that does not create multiple edges.
By applying the Jordan curve theorem we can show that
for the sphere the operation of adding an edge is always possible for a map
which is not a triangulation.

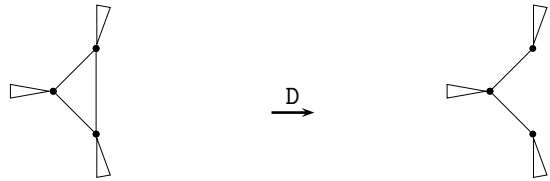
\begin{figure}
\centering
\psset{unit=.009\textwidth}
\parbox{.39\textwidth}{%
\centering
\ttfamily
\begin{pspicture}(-5,-5)(35,20)
\psline[linewidth=.3pt]{*-*}(25,5)(25,15)
\psline[linewidth=.3pt]{*-*}(20,10)(25,5)
\psline[linewidth=.3pt]{*-*}(20,10)(25,15)
\rput{80}(25,15){\pspolygon[linewidth=.3pt](.5,0)(5,.8)(5,-.8)}
\rput{180}(20,10){\pspolygon[linewidth=.3pt](.5,0)(5,.8)(5,-.8)}
\rput{280}(25,5){\pspolygon[linewidth=.3pt](.5,0)(5,.8)(5,-.8)}
\end{pspicture}}
\parbox{.19\textwidth}{%
\centering
\ttfamily
\begin{pspicture}(0,0)(15,20)
\psline[linewidth=1pt]{->}(5,10)(10,10)
\rput(7.5,12){{D}}
\end{pspicture}}
\parbox{.39\textwidth}{%
\centering
\ttfamily
\begin{pspicture}(-5,-5)(35,20)
\psline[linewidth=.3pt]{*-*}(5,10)(10,5)
\psline[linewidth=.3pt]{*-*}(5,10)(10,15)
\rput{80}(10,15){\pspolygon[linewidth=.3pt](.5,0)(5,.8)(5,-.8)}
\rput{180}(5,10){\pspolygon[linewidth=.3pt](.5,0)(5,.8)(5,-.8)}
\rput{280}(10,5){\pspolygon[linewidth=.3pt](.5,0)(5,.8)(5,-.8)}
\end{pspicture}}
\caption{Expansion for removing edge}
\label{removecorner}
\end{figure}

We would like to generate maps on other surfaces in a similar way.
However, there are maps which are not triangulations for which it is not 
possible to add an edge without producing multiple edges.
In Section~\ref{contractedges} we introduce a class of maps, which we called 
irreducible maps,
which are analogous to irreducible triangulations for the purpose of 
generating maps.

The generation of the maps with $n$ vertices of a fixed surface consists of
four steps:

\begin{enumerate}
\item Generate the {\em irreducible triangulations} of the surface 
(Section~\ref{findirrtri}).
\item Generate the {\em irreducible maps} of the surface from the 
irreducible triangulations by removing vertices (Section~\ref{findirrmaps}).
\item Split vertices (E-expansions) of the irreducible maps 
to obtain {\em face irreducible maps} with $n$ vertices (Figure~\ref{split} 
and Section~\ref{contractedges}).
\item Remove edges (D-expansions) of the face irreducible maps 
while the maps remain {\em $2$-cell embeddings} (Figure~\ref{removecorner} 
and Section~\ref{addedges}).
\end{enumerate}

Before describing these steps in 
Sections~\ref{addedges}--\ref{findirrtri} we give definitions
related to graphs in Section~\ref{definitions} and 
we provide notation used when describing operations for generating maps in
Section~\ref{recursive}.
The computer program {\em surftri} \cite{surftri} is an implementation of the 
techniques for generating triangulations and maps
described in this paper.
Information on how this was done is contained in Section~\ref{implement}.
In Section~\ref{facewidth} we discuss how the steps for generating maps can be
modified to generate closed $2$-cell embeddings and polyhedral embeddings.
We display the irreducible maps on the projective plane and torus in 
Section~\ref{display}.

\section{Definitions}
\label{definitions}

A {\em surface} is a two-dimensional compact manifold.
We denote the orientable surface with genus $g$, the sphere with $g$ handles
attached, as $S_g$ and the 
nonorientable surface with genus $g$, the sphere with $g$ crosscaps attached,
as $N_g$.

We consider only simple graphs which are graphs with no loops and no multiple 
edges.
Let $G$ be a connected graph embedded on a surface $S$.
A {\em face} of $G$ is 
a connected component of the complement of $G$ in $S$.
The graph $G$ is a {\em map} (or {\em open $2$-cell embedding}) on $S$ if 
every face of $G$ is homeomorphic to an open disk.
If the three edges $v_1v_2$, $v_2v_3$, and $v_3v_1$ are contained in the 
map $G$
then the union of these three edges is a {\em $3$-cycle} of $G$ denoted as
$v_1v_2v_3$.
A map $G$ on a surface $S$ is a {\em triangulation} of $S$ if the boundary of 
every face of $G$ is a $3$-cycle and the map is not a single $3$-cycle embedded
on the sphere.

Let $F$ be a face of $G$ with the boundary edges, in order, 
$v_1 v_2$, $v_2 v_3$, \dots, $v_m v_1$.  
The face $F$ is denoted by the list of
vertices $v_1 v_2 v_3 \dots v_m$.  These vertices do not need to be distinct.
We call $F$ an {\em m-face} and we say that $F$ has {\em size} $m$.
If $F$ has size $m \ge 4$ then~$F$ is defined to be a {\em large face}.
The subscripts of the vertices of a face are modulo $m$.
We say we {\em triangulate} a large face $F$ when we add vertices and edges 
to the interior of $F$ in such a way that all the new faces formed are 
$3$-faces.

\section{Recursive generation}
\label{recursive}

The basic ``isomorph-free'' generation technique that is used for the 
Steps~(1)--(4) is described in detail in 
\cite{MR2357364,plantripaper,MR1606516}.
For each generation process we specify the class~$\mathcal{C}$ which is the 
class of maps to be generated,
the initial class $\mathcal{C}_0 \subseteq \mathcal{C}$ 
which is the class of maps 
from which the maps in $\mathcal{C}$ are generated, 
and $F$ an expansion operation.
The expansion operation $F$ is a function from $\mathcal{C}$ into 
the set of subclasses of $\mathcal{C}$.
We say $(\mathcal{C}_0;F)$ generates $\mathcal{C}$ if for each 
$G \in \mathcal{C}$ there is a sequence $G_0$, $G_1$, \dots, $G_m = G$ 
such that $G_0 \in \mathcal{C}_0$ and for every $i$, $1 \le i \le m$ we have 
$G_i \in F(G_{i-1})$.
We call the expansion operation $F$ the {\em F-expansion}.
The inverse of the expansion operation $F$ is the {\em F-reduction}.

Figures~\ref{split} and \ref{removecorner} represent expansion operations
which we use to generate maps.
The left side of each figure shows a part of the embedded graph before the
expansion operation.
The right side shows the same part of the graph after the expansion operation.
Each part of the graphs shown
is contained in a simply connected component of the surface.
The full edges which are shown
are required to be a part of the map being modified.
The shorter half edges (in Figure~\ref{split}) are other unchanged edges of 
the map.  
The small flattened triangles (in Figures~\ref{split} and \ref{removecorner})
represent the location of zero or more other unchanged edges.
The expansion operation replaces the subgraph on the left with the subgraph on
the right.  The reduction operation replaces the right subgraph
with the left subgraph.
Since the part of the graph being modified is contained in a simply connected
component the surface remains the same.

\section{Removing edges in corners of large faces}
\label{addedges}

We examine the steps of the overall map generation process in reverse order 
to help clarify the choice of operations and initial classes which we use.

In Step 4
we generate $\mathcal{M}_2$, the class of all maps for a surface $S$.
We consider an arbitrary map of the type being generated and describe
the reduction operation, the D-reduction.
The D-expansion is the inverse of the D-reduction.
Let $\mathcal{M}_1 \subseteq \mathcal{M}_2$ be the class of 
those maps on $S$ for which the D-reduction is not possible. 
We characterize $\mathcal{M}_1$ and
show that $(\mathcal{M}_1;D)$ generates $\mathcal{M}_2$.

Let $G$ be a map on a surface $S$.
We continue to call the map $G$ even as it is modified by the D-reduction.
The D-expansion which is shown in Figure~\ref{removecorner}
is the removal of an edge from a $3$-face.
It can only be applied if the map has a $3$-face.
The D-reduction is the addition of an edge in the ``corner'' of a large face.
We only apply the D-reduction if the edge to be added does not already exist 
in the map.

Let $F = v_1 v_2 \dots v_m$  be a large face of $G$
and let $v_i$ be a vertex on $F$.
If  $v_i = v_{i+2}$ then $v_{i+1}$ has
degree $1$ and is not adjacent to $v_{i+3}$. The edge $v_{i+1} v_{i+3}$
can be added in the interior of~$F$
dividing $F$ into a $3$-face and a face of size $m{-}1$.
Adding this edge reduces the number of vertices of degree $1$.  
Repeated addition of edges of this type results in no faces having a vertex of 
degree $1$.

We can now assume that the three vertices 
$v_i, v_{i+1}, v_{i+2}$ on $F$ are distinct.
If $v_i$ is not adjacent to $v_{i+2}$ then again the edge $v_i v_{i+2}$
can be added in the interior of~$F$ 
dividing $F$ into a $3$-face and a face of size $m{-}1$.

The addition of an edge in this way is the D-reduction.
Even though a $D$-reduction adds an edge to the map it simplifies the 
map by making it more like a triangulation by increasing the number of $3$-faces.
The application of a $D$-reduction reduces the size of one large face and 
does not change the size of the other large faces.
So eventually no more $D$-reductions are possible.

If $S$ is the sphere then a $D$-reduction can always be applied in a
large face.
The initial class $\mathcal{M}_1$ of maps on the sphere for which 
the $D$-reduction 
is not possible is the class of triangulations of the sphere.

If $S$ is a surface other than the sphere then it is possible to
have a large face in which it is not possible to use a $D$-reduction.
For example, $P^4$, in Figure~\ref{mapsproj} shows $K_4$ 
embedded in the projective plane.  
Opposite points on the hexagon have been identified.
The map has three $4$-faces.
None of the faces are $3$-faces and no edges can be added since the graph is 
complete.

A large face $F = v_1 v_2 \dots v_m$ of a map is 
an {\em irreducible face} if for every $i$, $1 \le i \le m$,
the vertices $v_i$ and $v_{i+2}$ are adjacent.  
A map is {\em face irreducible} if every face is an irreducible face or 
a $3$-face.
Trivially, every triangulation is a face irreducible map.

So the initial class $\mathcal{M}_1$ of maps on $S$ 
for which the $D$-reduction 
is not possible is the class of face irreducible maps on $S$.

\section{Splitting vertices}
\label{contractedges}

We now consider Step 3 of the map generation process in which we generate the 
face irreducible maps.
The class of maps to be generated is $\mathcal{M}_1$ which is used as the 
initial class in the previous section.
The $E$-expansion shown in Figure~\ref{split} is used.
Below we specify an initial class $\mathcal{M}_0$ such that 
$(\mathcal{M}_0;E)$ generates $\mathcal{M}_1$.

Let $G \in \mathcal{M}_1$ be a face irreducible map on the surface $S$. 
The $E$-reduction, {\em edge contraction}, is the inverse of the 
$E$-expansion.
The $E$-reduction is applied only if the faces on both sides of the edge 
being contracted are $3$-faces.
Also the $E$-reduction is performed only if the resulting graph is still simple.
Let $v_1$ and~$v_2$ be the vertices of the edge to be contracted and 
let $u_1 v_1 v_2$ and $u_2 v_1 v_2$ be the $3$-faces on either side of $v_1 v_2$.
The two ends of the contracted edge, $v_1$ and $v_2$, must not both be 
adjacent to any vertices other than $u_1$ and $u_2$.  Otherwise, multiple edges
would be produced when $v_1 v_2$ is contracted.
An edge is {\em contractible} if it is on exactly two $3$-cycles both of which 
are $3$-faces.
To apply the edge contraction operation to an edge the edge must be 
contractible and the map must not be $K_4$ embedded in the sphere.
An edge is {\em noncontractible} if there is at least one $3$-cycle 
containing the edge which is not a $3$-face.
An edge is {\em essentially noncontractible} if at least one $3$-cycle 
containing the edge is an essential $3$-cycle on $S$.

\begin{theorem}
\label{IrreducibleFaces}

Every edge on the boundary of an irreducible face is 
essentially noncontractible.

\end{theorem}

\begin{proof}
Let $F = v_1 v_2 \dots v_m$ be an irreducible face and let $v_i v_{i+1}$ be 
an edge on the boundary of $F$.
Define $C_{F,i}$ to be the $3$-cycle $v_i v_{i+1} v_{i+2}$.
Assume $C_{F,i}$ is not essential, i.e. 
the interior of $C_{F,i}$ is simply connected.
Let $D$ be the disk consisting of $C_{F,i}$ and its interior.
The path $v_i v_{i+1} v_{i+2}$ is on the boundary of both $F$ and $D$.
The graph consisting of the vertices and edges in $D$ is a map on the sphere so
$F$ cannot be in $D$.
The vertex $v_{i+3}$ must be adjacent to $v_{i+1}$ but the edge 
$v_{i+3} v_{i+1}$ cannot be in the interior of $F$ or in the interior of $D$.
So $v_{i+3} = v_i$.
Similarly, $v_{i+2} = v_{i-1}$.
Thus the edge $e = v_{i+2} v_{i+3} = v_{i-1} v_i = v_{i+2} v_i$ 
must occur twice on $F$ and
$F$ is on both sides of $e$.
This is impossible since $C_{F,i}$ must be on one side of~$e$. 
\end{proof}

Contracting edges of $G$ does not change the surface $S$ 
in which $G$ is embedded so $C_{F,i}$ in the proof of
Theorem~\ref{IrreducibleFaces} remains essential and cannot 
become a $3$-face when edges are contracted.
So an irreducible face of $G$ remains an irreducible face of $G$ with the same
size when edges of $G$ are contracted.
As edges of $G$ are contracted the number of irreducible faces and the their 
sizes remain unchanged.

We apply the $E$-reduction while contractible edges remain in $G$.
Each application of the $E$-reduction reduces the number of vertices of $G$.
So after a finite number of $E$-reductions there are no contractible edges.

We define a map $G$ as an {\em irreducible map} if $G$ is 
face irreducible and no edge in~$G$ is contractible.
The following property of irreducible maps is a generalization of a similar
property of irreducible triangulations.

\begin{theorem}
\label{IrreducibleMap}

Every edge of an irreducible map is essentially noncontractible.

\end{theorem}

\begin{proof}

Let $G$ be an irreducible map.
An edge of $G$ is either on the boundary of an irreducible
face or on two $3$-faces.
In the first case the edge is essentially noncontractible by
Theorem~\ref{IrreducibleFaces}.
In the second case since the edge is not contractible it must be on a $3$-cycle 
$C$ which is not a $3$-face.
We show that $C$ is essential.
Assume $C$ is not essential.
On the surface $C$ bounds a disk $D$.
A new map $H$ on the sphere can be obtained by replacing the exterior of $D$ 
with a $3$-face.

The map $H$ cannot contain a large face.
If $H$ contains a large face $F = v_1 v_2 \dots v_m$ then at least one vertex,
say $v_2$ of $F$ must be in the interior of the (assumed) non-essential 
$3$-cycle $C$.
Then the essential $3$-cycle $v_1 v_2 v_3$ is in the disk $D$.  
But this is not possible.

The map $H$ contains no large face and is a triangulation of the sphere.
If $H$ is~$K_4$ then any interior edge of $D$ is contractible in $G$.
If $H$ is not $K_4$ then there are contractible edges of $H$ which are also 
contractible in $G$. 

\end{proof}

The initial class $\mathcal{M}_0$ is the class of all irreducible maps on 
the surface 
$S$ and $(\mathcal{M}_0;E)$ generates $\mathcal{M}_1$, the class of all
face irreducible maps on the surface $S$.

\section{Generating irreducible maps}
\label{findirrmaps}

Irreducible maps which we generate in Step 2 
have a very nice property which makes them easy to obtain.
Each irreducible map $G$ is a ``submap'' of an irreducible triangulation $T$, 
i.e. $G$ is obtained by removing zero or more vertices and adjacent edges
from~$T$.

The class to be generated is the class of irreducible maps.
The initial class is the class of irreducible triangulations.
The expansion operation is to remove a set of vertices and adjacent edges
with the condition 
that the resulting embedded graph is in the class of irreducible maps.
Theorem~\ref{TheoremIrreducibleMaps} shows that this expansion operation does
generate all of the irreducible maps.

The following theorem proves another property of irreducible faces and provides
us with the tool to prove Theorem~\ref{TheoremIrreducibleMaps}.

\begin{theorem}
\label{TriangulateIrrFace}

Let $F$ be an irreducible face of a map $G$ (which need not be irreducible).
We may replace the interior of $F$ with new vertices, edges, and faces 
to create a new map $G'$ such that
the new faces are 3-faces, 
the new edges are noncontractible in $G'$, and
the interior of the union of the new vertices, edges, and faces is simply 
connected.

\end{theorem}

\begin{proof}

There is one special irreducible face which we handle separately.
$P^3$ shown in Figure~\ref{mapsproj} is a single $3$-cycle embedded in the 
projective plane.
The dotted lines and open circles in the figure represent the edges and 
vertices which are added to obtain the irreducible triangulation $P^1$.
So we now can assume that $F$ is not the one face of~$P^3$.

Let $F = v_1 v_2 \dots v_m$ be an irreducible face of $G$.
We could triangulate $F$ by placing an $m$-cycle $w_1 w_2 \dots w_m$
in the interior of $F$ and an additional vertex $x$ inside this $m$-cycle.
Adding the edges $w_i v_i, w_i v_{i+1}, w_i x$ for $i$, $1 \le i \le m$ would
fill~$F$ with $3$-faces.
We could then contract edges in the interior of $F$ until no more edges in
the interior of $F$ are contractible.
However, this map might not contain $G$.
This might occur if, while contacting edges in the interior of $F$, 
two vertices of the boundary of $F$ are the ends of an edge which is
contracted thus merging these two vertices of $G$.
So triangulating of~$F$ must be done with more care.

We use the $V$-reduction to add one vertex at a time to the interior of $F$ 
and to attach the new vertex to vertices which are on the boundary of the 
large face.
The $V$-reduction is shown in Figure~\ref{removevertex4}.
This figure is similar to the previous figures showing expansions and 
reductions.
Figure~\ref{removevertex4} shows only four vertices of the large face, more
vertices may be used.
The figure also shows dashed curves.
Each dashed curve represents part of an edge which is not
completely contained in the simply connected component of the surface
represented by the figure.
One requirement for the application of the $V$-reduction is that
at least four edges are added joining consecutive vertices on the boundary of 
the face to the new vertex.
With each $V$-reduction the size of the large face is reduced.
We can apply the $V$-reduction only a finite number of times.
Let the new vertex be~$w$ and let these consecutive vertices on the 
boundary of the face be $v_1$, $v_2$, \dots, $v_k$ with $k \ge 4$.
A second requirement for the application of the $V$-reduction is 
that the $3$-cycles $v_i v_{i+1} v_{i+2}$ be essential for $i$,
$1 \le i \le k-2$.
Since the $3$-cycle $w v_i v_{i+2}$ is homeomorphic in $S$ to the
essential $3$-cycle $v_{i+1} v_i v_{i+2}$ for $i$,
$1 \le i \le k-2$ each new edge $w v_i$ is essentially noncontractible for $i$,
$1 \le i \le k$.

\begin{figure}
\centering
\psset{unit=.009\textwidth}
\parbox{.39\textwidth}{%
\centering
\ttfamily
\begin{pspicture*}(5,-10)(35,25)
\rput{180}(20,4){\pspolygon[linewidth=.3pt](.5,0)(4,.6)(4,-.6)}
\rput{225}(20,4){\pspolygon[linewidth=.3pt](.5,0)(4,.6)(4,-.6)}
\psline[linewidth=.3pt](20,4)(16,2.25)
\psline[linewidth=.3pt]{*-*}(20,4)(25,0)
\rput{270}(25,0){\pspolygon[linewidth=.3pt](.5,0)(4,.6)(4,-.6)}
\rput{225}(25,0){\pspolygon[linewidth=.3pt](.5,0)(4,.6)(4,-.6)}
\psline[linewidth=.3pt](25,0)(23.5,-3.75)
\psline[linewidth=.3pt](25,0)(28,0)

\rput{180}(20,11){\pspolygon[linewidth=.3pt](.5,0)(4,.6)(4,-.6)}
\rput{135}(20,11){\pspolygon[linewidth=.3pt](.5,0)(4,.6)(4,-.6)}
\psline[linewidth=.3pt](20,11)(16,12.75)
\psline[linewidth=.3pt]{*-*}(20,11)(25,15)
\rput{90}(25,15){\pspolygon[linewidth=.3pt](.5,0)(4,.6)(4,-.6)}
\rput{135}(25,15){\pspolygon[linewidth=.3pt](.5,0)(4,.6)(4,-.6)}
\psline[linewidth=.3pt](25,15)(23.5,18.75)
\psline[linewidth=.3pt](25,15)(28,15)

\psline[linewidth=.3pt]{*-}(28,7.5)(20,4)
\psline[linewidth=.3pt]{*-}(28,7.5)(25,0)
\psline[linewidth=.3pt]{*-}(28,7.5)(20,11)
\psline[linewidth=.3pt]{*-}(28,7.5)(25,15)
\psline[linewidth=.3pt](20,4)(20,11)

\psarc[linestyle=dashed,dash=3pt 3pt,linewidth=.3pt](16,3){10}{95}{315}
\psarc[linestyle=dashed,dash=3pt 3pt,linewidth=.3pt](16,12){10}{45}{265}

\end{pspicture*}}
\parbox{.19\textwidth}{%
\centering
\ttfamily
\begin{pspicture}(0,0)(15,20)
\psline[linewidth=1pt]{->}(5,10)(10,10)
\rput(7.5,12){{V}}
\end{pspicture}}
\parbox{.39\textwidth}{%
\centering
\ttfamily
\begin{pspicture*}(5,-10)(35,25)
\rput{180}(20,4){\pspolygon[linewidth=.3pt](.5,0)(4,.6)(4,-.6)}
\rput{225}(20,4){\pspolygon[linewidth=.3pt](.5,0)(4,.6)(4,-.6)}
\psline[linewidth=.3pt](20,4)(16,2.25)
\psline[linewidth=.3pt]{*-*}(20,4)(25,0)
\rput{270}(25,0){\pspolygon[linewidth=.3pt](.5,0)(4,.6)(4,-.6)}
\rput{225}(25,0){\pspolygon[linewidth=.3pt](.5,0)(4,.6)(4,-.6)}
\psline[linewidth=.3pt](25,0)(23.5,-3.75)
\psline[linewidth=.3pt](25,0)(28,0)

\rput{180}(20,11){\pspolygon[linewidth=.3pt](.5,0)(4,.6)(4,-.6)}
\rput{135}(20,11){\pspolygon[linewidth=.3pt](.5,0)(4,.6)(4,-.6)}
\psline[linewidth=.3pt](20,11)(16,12.75)
\psline[linewidth=.3pt]{*-*}(20,11)(25,15)
\rput{90}(25,15){\pspolygon[linewidth=.3pt](.5,0)(4,.6)(4,-.6)}
\rput{135}(25,15){\pspolygon[linewidth=.3pt](.5,0)(4,.6)(4,-.6)}
\psline[linewidth=.3pt](25,15)(23.5,18.75)
\psline[linewidth=.3pt](25,15)(28,15)

\psline[linewidth=.3pt](20,4)(20,11)

\psarc[linestyle=dashed,dash=3pt 3pt,linewidth=.3pt](16,3){10}{95}{315}
\psarc[linestyle=dashed,dash=3pt 3pt,linewidth=.3pt](16,12){10}{45}{265}

\end{pspicture*}}

\caption{Expansion for removing vertex, $k = 4$}
\label{removevertex4}
\end{figure}
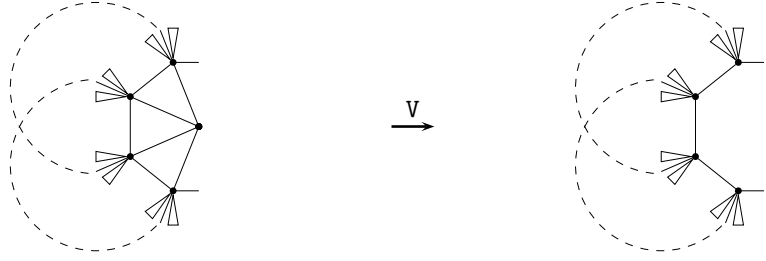

We now describe how to add the first vertex to the interior of $F$.

First assume  the vertices $v_1$, $v_2$, \dots, $v_m$ of $F$ are distinct.
This is always the case when $m \le 5$ since each vertex on $F$ is adjacent 
to all the other vertices on~$F$.
Since~$F$ is an irreducible face and $m \ge 4$ we use all the vertices of
$F$ for the $V$-reduction.
The added edges are all essentially noncontractible and all the new faces
are 3-faces satisfying the theorem.

Now assume the vertices of $F$ are not distinct. 
There is at least
one vertex which occurs at least twice on the boundary of~$F$.
To triangulate the interior of $F$ we must add
more than one new vertex to the interior of $F$ to prevent 
multiple edges.

Let $k$ be the number of vertices in the longest paths on the boundary of~$F$.
Since there is at least
one vertex which occurs at least twice on the boundary of~$F$ we have 
$k < m$.
If necessary we relabel the vertices so that
$p = v_1 v_2 \dots v_k$ is one of the longest paths on the 
boundary of $F$.
The vertices $v_1$, $v_2$, \dots, $v_k$ are all distinct. 
The vertex $v_m$ occurs at least twice on $F$ and is on $p$ for, 
otherwise, we could extend~$p$  
to $v_m$ producing a longer path on the boundary of $F$.
Since $v_m$ is adjacent to $v_1$ and $v_2$, $v_m = v_a$ for some $a$, 
$3 \le a \le k$.
Likewise, $v_{k+1} = v_b$ for some $b$, $1 \le b \le k-2$.

We assert that $k \ge 4$.
Since $F$ is an irreducible face $v_i$, $v_{i+1}$, and $v_{i+2}$ are pairwise
adjacent and distinct for every $i$, $1 \le i \le m$.  
So $k \ge 3$.  Suppose that $k = 3$.
Then for every $i$, $1 \le i \le m$ we have $v_i = v_{i+3}$.  
The edges of $F$ would be 
$v_1 v_2$, $v_2 v_3$, $v_3 v_1$, $v_1 v_2$, $v_2 v_3$, $v_3 v_1$, \dots.
Since a face can occur on an edge at most twice $F = v_1 v_2 v_3 v_1 v_2 v_3$
and the map is $P_3$ contrary to our earlier assumption.

We use $\{v_1, v_2, \dots, v_k\}$ to 
apply a $V$-reduction and call the new vertex $w_1$.
The resulting large face $F_1$ is $v_1 w_1 v_k \dots v_m$.
We observe that the face $F_1$ might not be irreducible but it almost 
satisfies the definition.
For every $i$, $k \le i \le m-1$, $v_i$ is adjacent to $v_{i+2}$.
Since $v_m = v_a$ and $v_{k+1} = v_b$, 
$w_1$ is adjacent to $v_m$ and $v_{k+1}$.
The only possible missing condition is that $v_1$ might not be adjacent 
to $v_k$.

In the remainder of the proof we add additional vertices 
$w_2$, $w_3$, \dots, in a similar way.
With each additional vertex $w_{n+1}$ a face $F_{n+1}$ 
is produced which has fewer
vertices of the original face $F$ than $F_n$ has.
When a face $F_{n+1}$ with no vertices of the original face $F$ is obtained
then we finish triangulating $F$ as described below.

Assume we have added $n$ new vertices to the interior of $F$ and we
have a large face $F_n$ with boundary 
$v_1 w_1 \dots w_n v_j \dots v_m$ such that $j \le m$ and
$w_n$ is adjacent to $v_{j+1}$ (which is $v_1$ if $j = m$).
We have shown above that this assumption is true for $n=1$.
Starting from this assumption for $n$ we show either 
({\em i}) that we can finish triangulating $F$ or 
({\em ii}) that the assumption is true for $n+1$ and the face $F_{n+1}$ 
has fewer vertices of the original face $F$ than $F_n$ has.
Thus the construction terminates in a finite number of steps.

We obtain ({\em i})
when the vertices on the boundary of $F_n$ are distinct.
In this case, the vertices for attaching $w_{n+1}$ are 
$\{w_n, v_j, \dots, v_m, v_1, w_1\}$.
Recall that when $w_1$ is attached $w_1$ is adjacent to $v_m$.
If $n+1 \le 3$ then there is no longer a large face in $F$ and all the edges
which have been added in the interior of $F$ are essentially noncontractible.
If $n+1 > 3$ then there is a resulting large face $F_{n+1}$ with boundary
$w_1 w_2 \dots w_{n+1}$.
We arbitrarily triangulate the face $F_{n+1}$ with edges 
$w_{n+1} w_i$ for $i$, $1 < i < n$.
Some of these $n-2$ edges might be contractible but all the other 
edges which have been added to the interior of $F$ 
are essentially noncontractible.
We repeatedly contract any contractible edge in the interior of $F_{n+1}$ 
until there are no contractible edges in the interior of $F_{n+1}$.
All of the edges in the interior of $F$ are then noncontractible
and we have triangulated the face $F$ as required.

We can show ({\em ii}) when the vertices on the boundary of $F_n$ are 
not distinct.
The vertices $w_n$, $v_j$, $v_{j+1}$, and $v_{j+2}$ are distinct 
since $w_n$ is not on
the boundary of $F$ and the vertices $v_j$, $v_{j+1}$, and $v_{j+2}$ are 
pairwise adjacent.
Let $v_{j'}$ be the vertex on the boundary of $F_n$ such that 
$w_n$, $v_j$, $v_{j+1}$, $v_{j+2}$, \dots, $v_{j'}$ are distinct 
and $v_{j'+1}$ is in
$\{v_j, v_{j+1}, v_{j+2}, \dots, v_{j'-2}\}$.
We use $\{w_n, v_j, v_{j+1}, v_{j+2}, \dots, v_{j'}\}$ to apply a 
$V$-reduction with the new vertex $w_{n+1}$.
In this way we obtain a smaller face $F_{n+1}$ which fulfills our assumption
for $n+1$ which is ({\em ii}). 
\end{proof}

\begin{theorem}
\label{TheoremIrreducibleMaps}

If $G$ is an irreducible map on a surface $S$ then there is at least
one irreducible triangulation $T$ of $S$ from which $G$ may be 
obtained by removing a set of vertices
from $T$ along with the edges containing these vertices.

\end{theorem}

\begin{proof}
Let $G$ be the irreducible map.
Using Theorem~\ref{TriangulateIrrFace} we ``irreducibly triangulate'' each 
irreducible face of $G$ to produce a triangulation $T$.
The edges of $T$ which are in $G$ are essentially noncontractible by 
Theorem~\ref{IrreducibleMap}.
The edges of $T$ which are not in $G$ are noncontractible by
Theorem~\ref{TriangulateIrrFace}.
It may be possible to triangulate the irreducible faces of $G$ 
in more than one way
so the irreducible triangulation $T$ may be one of many which 
satisfy the theorem. 
\end{proof}

\section{Generating irreducible triangulations}
\label{findirrtri}

Irreducible triangulations which are generated in Step 1 have been extensively
studied.
For any fixed surface the number of irreducible triangulations is 
finite~\cite{MR1021367}.
Irreducible triangulations have been determined and displayed by a number of 
authors:
the single irreducible triangulation of the sphere ($S_0$) by 
Steinitz and Rademacher~\cite{StRa};
the two irreducible triangulations of the projective plane or the cross 
surface ($N_1$) by Barnette \cite{MR84f:57009}; 
the 21 irreducible triangulations of the torus ($S_1$) by 
Lawrencenko \cite{MR914777});
and the $29$ irreducible triangulations of the Klein bottle ($N_2$) 
by Lawrencenko and Negami \cite{MR98h:05067} and 
Sulanke \cite{math.CO/0407008}.
The irreducible triangulations of the double torus ($S_2$), 
the triple cross surface ($N_3$), and the quadruple cross surface ($N_4$)
have been generated by the author using an extension of computer program 
{\em surftri} \cite{surftri}.
The counts of irreducible triangulations are shown in 
Table~\ref{Counts}.
The largest of these classes ($N_4$) required 54 CPU days.
The author estimates it would take CPU centuries for this program 
to generate the irreducible triangulations for $S_3$ or $N_5$.

We describe briefly how the two stage generation process for irreducible 
triangulations of a surface $S$ works.
In the first stage we use the vertex splitting operation 
to generate triangulations on slightly simpler surfaces than $S$.
We impose certain conditions necessary for the second stage on these 
triangulations.
These conditions limit the triangulations generated to a finite number.
In the second stage these triangulations are modified in such a way that new
handles or crosscaps are added to produce irreducible triangulations on $S$.
Only the reduction operation for the second stage is described.
More details may be found in \cite{gentriang}.

Let $S$ (not the sphere) be the surface for which we are generating 
irreducible triangulations.
Let $G$ be an irreducible triangulation of $S$. 
Theorem~\ref{transverse} below shows that $G$ contains many nonseparating 
$3$-cycles.
Let $w_1 w_2 w_3$ be a nonseparating $3$-cycle of $G$.
We create a new triangulation $G'$ of a different surface $S'$ 
using the operations described below.

If the $3$-cycle $w_1 w_2 w_3$ in $G$ is two-sided 
then we cut along $w_1 w_2 w_3$ to produce
a surface $S'$ with a boundary consisting of two disjoint $3$-cycles 
$u'_1 u'_2 u'_3$
and $v'_1 v'_2 v'_3$ where $u'_i$ and $v'_i$ come 
from the original vertex $w_i$ for $i=1,2,3$. 
We cap the holes with two $3$-faces $u'_1 u'_2 u'_3$ and $v'_1 v'_2 v'_3$.
$G'$ is now a triangulation of $S'$ which has an Euler genus two less than
the Euler genus of $S$.

If the $3$-cycle $w_1 w_2 w_3$ in $G$ is one-sided 
then we cut along $w_1 w_2 w_3$ to produce
a surface $S'$ with a boundary consisting of the 6-cycle 
$u'_1 u'_2 u'_3 v'_1 v'_2 v'_3$ where $u'_i$ and $v'_i$
again come from the original vertex $w_i$ for $i=1,2,3$. 
We cap the hole with a new vertex $t'$ and six $3$-faces 
$t' u'_1 u'_2$, $t' u'_2 u'_3$, $t' u'_3 v'_1$,
$t' v'_1 v'_2$, $t' v'_2 v'_3$, and $t' v'_3 u'_1$.
$G'$ is now a triangulation  of $S'$ which has an Euler genus one less than
the Euler genus of $S$.

The following theorem is similar to Lemma 4 of \cite{MR84f:57009} and 
Lemma 4 of \cite{MR914777}.
In a triangulation the {\em link} of a vertex $v$ is the cycle 
which is the boundary of the union of the faces containing $v$.

\begin{theorem}
\label{transverse}
Let $G$ be an irreducible triangulation of a surface other than the sphere,
let $v$ be a vertex of $G$, and let $L$ be the link of $v$.
Then there are two nonseparating $3$-cycles $v v_i v_k$ and $v v_j v_l$
such that $v_i$, $v_j$, $v_k$, and $v_l$ are distinct
and one path from $v_i$ to $v_k$ in $L$ contains $v_j$ and
the other path from $v_i$ to $v_k$ in $L$ contains $v_l$.
\end{theorem}

\begin{proof}
Since $G$ is irreducible,
for any vertex $u$ in $L$ the edge $v u$ is on a nonfacial $3$-cycle $v u w$.
Pick two vertices $v_i$ and $v_k$ in $L$ for which $v v_i v_k$ is a
nonfacial $3$-cycle and the distance from $v_i$ to $v_k$ in $L$ is minimal.
The shorter path from $v_i$ to $v_k$ in $L$ must have an interior vertex
since $v v_i v_k$ is not a face.
Let the vertex $v_j$ be such an interior vertex on the shorter path from $v_i$
to $v_k$ in $L$.
Let $v_l$ be a vertex in $L$ such that $v v_j v_l$ is a nonfacial $3$-cycle.
$v_l$ is not on the path from $v_i$ to $v_k$ in $L$ containing $v_j$ 
since the distance from $v_j$ and $v_l$ in $L$ is at least the distance 
from $v_i$ and $v_k$ in $L$.
Suppose $v v_i v_k$ separates the surface.  
Then $v_j$ and $v_l$ would be in different components but $v_j v_l$ is 
an edge.
Therefore, $v v_i v_k$ is nonseparating and, similarly, $v v_j v_l$ is also
nonseparating. 
\end{proof}

\section{Implementation}
\label{implement}

The computer program {\em surftri} \cite{surftri} implements the procedures 
to generate maps on various surfaces.
Many of the ideas and much of the code used in {\em surftri} are taken from 
the work of Brinkmann and McKay.
Their program {\em plantri} \cite{plantri} generates triangulations and maps
on the sphere as well as other classes of planar graphs.

The operation of {\em plantri} is described in \cite{MR2357364,plantripaper}.
To obtain triangulations of the 
sphere with $n$ vertices {\em plantri} starts with the only irreducible 
triangulation of the sphere, $K_4$, and
vertices are split using variations of the $E$-expansion until the 
triangulations have $n$ vertices.
If maps with $n$ 
vertices are being generated then as each triangulation with $n$ vertices is
produced the program switches to the mode of using the 
operation of removing edges. 
Edges are removed from $3$-faces using the $D$-expansion.

To generate triangulations of a surface in the program {\em surftri} we start
with the irreducible triangulations of that surface.
Vertices are split using the $E$-expansion to obtain triangulations of the
surface using procedures similar those used in {\em plantri}.
The list of the irreducible triangulations is provided as input to 
{\em surftri}.

\begin{table}
\centering
\caption{The number of irreducible triangulations and irreducible maps}
\begin{tabular}{r|r r r r r r}
                                    & $S_1$ & $S_2$ & $N_1$ & $N_2$ & $N_3$ & $N_4$ \\
\hline
Irreducible triangulations          & 21    &  396784 & 2   &  29   &  9708 &  6297982 \\
Irreducible maps                    & 68    & 2181071 & 7   & 173   & 75596 & 62641140 \\
\end{tabular}
\vspace{.1in}
\label{Counts}
\end{table}

To generate maps on a surface with $n$ vertices {\em surftri} generates 
the face irreducible maps with $n$ vertices.
Vertex splitting, the $E$-expansion, is again used.
{\em surftri} starts with irreducible maps and
vertex splitting is only done if each irreducible face remains irreducible.
The irreducible maps are provided as input, having been pregenerated and stored
on disk.
As each face irreducible map with $n$ vertices is
produced {\em surftri} switches to the mode of using the 
operation of removing edges. 
Edges are removed from $3$-faces using the $D$-expansion.

The irreducible maps with at least one large face 
were generated using Theorem~\ref{TheoremIrreducibleMaps}
rather than the construction used in its proof.
For a fixed surface each irreducible triangulation
was processed by removing sets of 
vertices and checking if the results were irreducible maps.
Duplicates were removed by sorting all the irreducible maps obtained in this 
way.
The number of irreducible triangulations of a surface is 
finite~\cite{MR1021367} and each irreducible map is obtained by removing 
vertices from an irreducible triangulation.
Therefore, the number of irreducible maps on a surface is finite.
We also show the counts of irreducible maps in Table~\ref{Counts}.

\section{Maps, closed $2$-cell embeddings, and polyhedral embeddings}
\label{facewidth}

We have described the steps listed in Section~\ref{introduction} for 
generating maps (open $2$-cell embeddings).  
We now consider how more restricted classes of maps can be generated by 
modifying these steps.

The {\em face-width} of an embedded graph 
on a surface is the smallest number $k$ such that there is a noncontractible 
closed curve on the surface that intersects the graph at $k$ points 
\cite{MR961150}.

Maps are embedded graphs which have face-width at least $1$ and are 
$1$-connected.

A {\em closed $2$-cell embedding} is a map for which the closure of every 
face is a closed $2$-cell.
For every face $F$ of a closed $2$-cell embedding no vertex occurs more 
than once on $F$.
Closed $2$-cell embeddings are those maps which have face-width at least~$2$ 
and are $2$-connected  \cite{MR1844449}.
The D-expansion does not increase the face-width or the connectivity of a map.
So in order to obtain a closed $2$-cell embedding when we apply the D-expansion
we must apply the operation to another closed $2$-cell embedding.
However, the E-expansion may increase the face-width of some maps from $1$
to $2$.
We only modify the final step used to generate maps to obtain a 
procedure for generating closed $2$-cell embeddings.

The generation of the closed $2$-cell embeddings with $n$ vertices of a fixed 
surface consists of four steps:

\begin{enumerate}
\item Generate the {\em irreducible triangulations} of the surface.
\item Generate the {\em irreducible maps} of the surface from the 
irreducible triangulations by removing vertices.
\item Split vertices (E-expansions) of the irreducible maps
to obtain {\em face irreducible maps} with $n$ vertices.
\item Remove edges (D-expansions) of the face irreducible maps
while the maps remain {\em closed $2$-cell embeddings}.
\end{enumerate}

A {\em polyhedral embedding} is a map for which the closures of any pair of 
faces have exactly one vertex, exactly one edge, or no points in common.
Polyhedral embeddings on a surface are those maps which have face-width at 
least $3$ and are $3$-connected \cite{MR1844449}.

We use Theorem~\ref{IrreducibleFaces} to show that any map with an 
irreducible face has face-width at most $2$.
Let $F = v_1 v_2 \dots v_m$ be an irreducible face on a surface $S$.
We construct a closed curve consisting of two segments.
One segment is in the interior of $F$ connecting $v_1$ and $v_3$.
The other segment is close to the edge $v_1 v_3$ and connects them.
This closed curve is homeomorphic in $S$ to the
essential $3$-cycle $v_1 v_2 v_3$ and thus is noncontractible.
This closed curve intersects
the graph at only $2$ points, $v_1$ and $v_3$.

Again we note that the D-expansion does not increase the face-width 
nor the connectivity of a map of a map.
Also, the E-expansion is restricted from eliminating irreducible faces. 
So we do not need to use any maps with irreducible faces in the generation of 
polyhedral embeddings.

The generation of the polyhedral embeddings with $n$ vertices of a fixed 
surface consists of only three steps:

\begin{enumerate}
\item Generate the {\em irreducible triangulations} of the surface.
\item Split vertices (E-expansions) of the irreducible triangulations
to obtain {\em triangulations} with $n$ vertices.
\item Remove edges (D-expansion) of the triangulations
while the maps remain {\em polyhedral embeddings}.
\end{enumerate}

In Table~\ref{MapcountsProj} we show the counts of maps, 
closed $2$-cell embeddings, and polyhedral embeddings on the projective plane
for increasing numbers of vertices.
We also show the counts for irreducible maps and face irreducible maps.
The {\em surftri} program produced these values.
Table~\ref{MapcountsTorus} shows the counts for maps on the torus.

We estimate that the entries in the tables which are blank would require 
more than $100$ days of CPU time to compute on $2.4$ GHz processors.
We adapted the data structures used in the {\em surftri} program to store 
the embeddings from the data structures used in {\em plantri}.
We modified these data structures to allow embeddings 
in non-orientable surfaces.
These modified data structures require more computer 
operations than are used in {\em plantri}.
The generation rates for {\em surftri} range from $1$ to $1.4$ 
million maps/second on a $2.4$ GHz processor.
When generating maps on the sphere the rates for {\em surftri}
are $0.7$ to $0.95$ of those rates observed using {\em plantri}.

\section{Irreducible Maps on the projective plane and the torus}
\label{display}

We provide drawings of the irreducible maps on the projective plane and torus.
For each irreducible map with large faces we indicate 
one of the irreducible triangulations from which it may be obtained.
The vertices and edges which have been removed from the irreducible 
triangulation are shown as open circles and dotted lines.
The irreducible triangulation chosen requires the minimum number of vertices
to be removed.

Figure~\ref{irrproj} shows the two irreducible triangulations of the 
projective plane \cite{MR84f:57009}.
Figure~\ref{mapsproj} shows the five irreducible maps with large faces 
on the projective plane.

There are $21$ irreducible triangulations, $T^1$--$T^{21}$, 
of the torus~\cite{MR914777}
which are shown in Figures~\ref{irrtorus1} and \ref{irrtorus2}.
Figures~\ref{mapstorus1}-\ref{mapstorus4} show the 47 irreducible maps, 
$T^{22}$--$T^{68}$, with large faces on the torus.

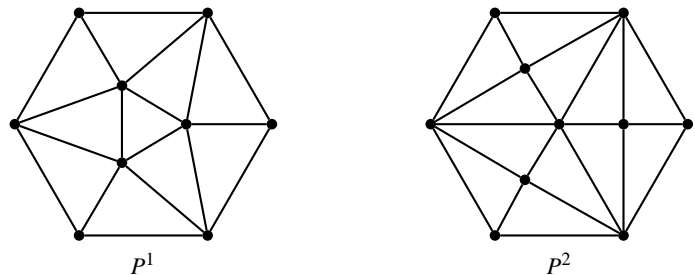
\begin{figure}[H]
\hfill
\psset{unit=.0045\textwidth}
\begin{pspicture}(-33,-36)(33,36)
\cnode*(30,0){2pt}{f1}
\cnode*(15,26){2pt}{e1}
\cnode*(-15,26){2pt}{d1}
\cnode*(-30,0){2pt}{f2}
\cnode*(-15,-26){2pt}{e2}
\cnode*(15,-26){2pt}{d2}
\cnode*(-5,-9){2pt}{a1}
\cnode*(-5,9){2pt}{b1}
\cnode*(10,0){2pt}{c1}
\ncline{f1}{e1}
\ncline{f1}{d2}
\ncline{f1}{c1}
\ncline{e1}{d1}
\ncline{e1}{b1}
\ncline{e1}{c1}
\ncline{d1}{f2}
\ncline{d1}{b1}
\ncline{f2}{e2}
\ncline{f2}{a1}
\ncline{f2}{b1}
\ncline{e2}{d2}
\ncline{e2}{a1}
\ncline{d2}{a1}
\ncline{d2}{c1}
\ncline{a1}{b1}
\ncline{a1}{c1}
\ncline{b1}{c1}
\rput[b](0,-35){\small $P^1$}
\end{pspicture}
\hfill
\psset{unit=.0045\textwidth}
\begin{pspicture}(-33,-36)(33,36)
\cnode*(30,0){2pt}{b1}
\cnode*(15,26){2pt}{e1}
\cnode*(-15,26){2pt}{d1}
\cnode*(-30,0){2pt}{b2}
\cnode*(-15,-26){2pt}{e2}
\cnode*(15,-26){2pt}{d2}
\cnode*(-8,-13){2pt}{a1}
\cnode*(15,0){2pt}{f1}
\cnode*(0,0){2pt}{c1}
\cnode*(-8,13){2pt}{g1}
\ncline{b1}{e1}
\ncline{b1}{d2}
\ncline{b1}{f1}
\ncline{e1}{d1}
\ncline{e1}{f1}
\ncline{e1}{c1}
\ncline{e1}{g1}
\ncline{d1}{b2}
\ncline{d1}{g1}
\ncline{b2}{e2}
\ncline{b2}{a1}
\ncline{b2}{c1}
\ncline{b2}{g1}
\ncline{e2}{d2}
\ncline{e2}{a1}
\ncline{d2}{a1}
\ncline{d2}{f1}
\ncline{d2}{c1}
\ncline{a1}{c1}
\ncline{f1}{c1}
\ncline{c1}{g1}
\rput[b](0,-35){\small $P^2$}
\end{pspicture}
\hfill\ 
\caption{Irreducible triangulations of the Projective Plane}
\label{irrproj}
\end{figure}

\begin{figure}[H]
\hfill
\psset{unit=.0045\textwidth}
\begin{pspicture}(-33,-36)(33,36)
\cnode*(30,0){2pt}{f1}
\cnode*(15,26){2pt}{e1}
\cnode*(-15,26){2pt}{d1}
\cnode*(-30,0){2pt}{f2}
\cnode*(-15,-26){2pt}{e2}
\cnode*(15,-26){2pt}{d2}
\cnode(-5,-9){2pt}{a1}
\cnode(-5,9){2pt}{b1}
\cnode(10,0){2pt}{c1}
\ncline{f1}{e1}
\ncline{f1}{d2}
\ncline[linestyle=dotted,linewidth=1.3pt]{f1}{c1}
\ncline{e1}{d1}
\ncline[linestyle=dotted,linewidth=1.3pt]{e1}{b1}
\ncline[linestyle=dotted,linewidth=1.3pt]{e1}{c1}
\ncline{d1}{f2}
\ncline[linestyle=dotted,linewidth=1.3pt]{d1}{b1}
\ncline{f2}{e2}
\ncline[linestyle=dotted,linewidth=1.3pt]{f2}{a1}
\ncline[linestyle=dotted,linewidth=1.3pt]{f2}{b1}
\ncline{e2}{d2}
\ncline[linestyle=dotted,linewidth=1.3pt]{e2}{a1}
\ncline[linestyle=dotted,linewidth=1.3pt]{d2}{a1}
\ncline[linestyle=dotted,linewidth=1.3pt]{d2}{c1}
\ncline[linestyle=dotted,linewidth=1.3pt]{a1}{b1}
\ncline[linestyle=dotted,linewidth=1.3pt]{a1}{c1}
\ncline[linestyle=dotted,linewidth=1.3pt]{b1}{c1}
\rput[b](0,-35){\small $P^3$}
\end{pspicture}
\hfill
\psset{unit=.0045\textwidth}
\begin{pspicture}(-33,-36)(33,36)
\cnode*(30,0){2pt}{b1}
\cnode*(15,26){2pt}{e1}
\cnode*(-15,26){2pt}{d1}
\cnode*(-30,0){2pt}{b2}
\cnode*(-15,-26){2pt}{e2}
\cnode*(15,-26){2pt}{d2}
\cnode(-8,-13){2pt}{a1}
\cnode(15,0){2pt}{f1}
\cnode*(0,0){2pt}{c1}
\cnode(-8,13){2pt}{g1}
\ncline{b1}{e1}
\ncline{b1}{d2}
\ncline[linestyle=dotted,linewidth=1.3pt]{b1}{f1}
\ncline{e1}{d1}
\ncline[linestyle=dotted,linewidth=1.3pt]{e1}{f1}
\ncline{e1}{c1}
\ncline[linestyle=dotted,linewidth=1.3pt]{e1}{g1}
\ncline{d1}{b2}
\ncline[linestyle=dotted,linewidth=1.3pt]{d1}{g1}
\ncline{b2}{e2}
\ncline[linestyle=dotted,linewidth=1.3pt]{b2}{a1}
\ncline{b2}{c1}
\ncline[linestyle=dotted,linewidth=1.3pt]{b2}{g1}
\ncline{e2}{d2}
\ncline[linestyle=dotted,linewidth=1.3pt]{e2}{a1}
\ncline[linestyle=dotted,linewidth=1.3pt]{d2}{a1}
\ncline[linestyle=dotted,linewidth=1.3pt]{d2}{f1}
\ncline{d2}{c1}
\ncline[linestyle=dotted,linewidth=1.3pt]{a1}{c1}
\ncline[linestyle=dotted,linewidth=1.3pt]{f1}{c1}
\ncline[linestyle=dotted,linewidth=1.3pt]{c1}{g1}
\rput[b](0,-35){\small $P^4$}
\end{pspicture}
\hfill
\psset{unit=.0045\textwidth}
\begin{pspicture}(-33,-36)(33,36)
\cnode*(30,0){2pt}{b1}
\cnode*(15,26){2pt}{e1}
\cnode*(-15,26){2pt}{d1}
\cnode*(-30,0){2pt}{b2}
\cnode*(-15,-26){2pt}{e2}
\cnode*(15,-26){2pt}{d2}
\cnode(-8,-13){2pt}{a1}
\cnode*(15,0){2pt}{f1}
\cnode*(0,0){2pt}{c1}
\cnode(-8,13){2pt}{g1}
\ncline{b1}{e1}
\ncline{b1}{d2}
\ncline{b1}{f1}
\ncline{e1}{d1}
\ncline{e1}{f1}
\ncline{e1}{c1}
\ncline[linestyle=dotted,linewidth=1.3pt]{e1}{g1}
\ncline{d1}{b2}
\ncline[linestyle=dotted,linewidth=1.3pt]{d1}{g1}
\ncline{b2}{e2}
\ncline[linestyle=dotted,linewidth=1.3pt]{b2}{a1}
\ncline{b2}{c1}
\ncline[linestyle=dotted,linewidth=1.3pt]{b2}{g1}
\ncline{e2}{d2}
\ncline[linestyle=dotted,linewidth=1.3pt]{e2}{a1}
\ncline[linestyle=dotted,linewidth=1.3pt]{d2}{a1}
\ncline{d2}{f1}
\ncline{d2}{c1}
\ncline[linestyle=dotted,linewidth=1.3pt]{a1}{c1}
\ncline{f1}{c1}
\ncline[linestyle=dotted,linewidth=1.3pt]{c1}{g1}
\rput[b](0,-35){\small $P^5$}
\end{pspicture}
\hfill\ 

\hfill
\psset{unit=.0045\textwidth}
\begin{pspicture}(-33,-36)(33,36)
\cnode*(30,0){2pt}{f1}
\cnode*(15,26){2pt}{e1}
\cnode*(-15,26){2pt}{d1}
\cnode*(-30,0){2pt}{f2}
\cnode*(-15,-26){2pt}{e2}
\cnode*(15,-26){2pt}{d2}
\cnode*(-5,-9){2pt}{a1}
\cnode*(-5,9){2pt}{b1}
\cnode(10,0){2pt}{c1}
\ncline{f1}{e1}
\ncline{f1}{d2}
\ncline[linestyle=dotted,linewidth=1.3pt]{f1}{c1}
\ncline{e1}{d1}
\ncline{e1}{b1}
\ncline[linestyle=dotted,linewidth=1.3pt]{e1}{c1}
\ncline{d1}{f2}
\ncline{d1}{b1}
\ncline{f2}{e2}
\ncline{f2}{a1}
\ncline{f2}{b1}
\ncline{e2}{d2}
\ncline{e2}{a1}
\ncline{d2}{a1}
\ncline[linestyle=dotted,linewidth=1.3pt]{d2}{c1}
\ncline{a1}{b1}
\ncline[linestyle=dotted,linewidth=1.3pt]{a1}{c1}
\ncline[linestyle=dotted,linewidth=1.3pt]{b1}{c1}
\rput[b](0,-35){\small $P^6$}
\end{pspicture}
\hfill
\psset{unit=.0045\textwidth}
\begin{pspicture}(-33,-36)(33,36)
\cnode*(30,0){2pt}{b1}
\cnode*(15,26){2pt}{e1}
\cnode*(-15,26){2pt}{d1}
\cnode*(-30,0){2pt}{b2}
\cnode*(-15,-26){2pt}{e2}
\cnode*(15,-26){2pt}{d2}
\cnode(-8,-13){2pt}{a1}
\cnode*(15,0){2pt}{f1}
\cnode*(0,0){2pt}{c1}
\cnode*(-8,13){2pt}{g1}
\ncline{b1}{e1}
\ncline{b1}{d2}
\ncline{b1}{f1}
\ncline{e1}{d1}
\ncline{e1}{f1}
\ncline{e1}{c1}
\ncline{e1}{g1}
\ncline{d1}{b2}
\ncline{d1}{g1}
\ncline{b2}{e2}
\ncline[linestyle=dotted,linewidth=1.3pt]{b2}{a1}
\ncline{b2}{c1}
\ncline{b2}{g1}
\ncline{e2}{d2}
\ncline[linestyle=dotted,linewidth=1.3pt]{e2}{a1}
\ncline[linestyle=dotted,linewidth=1.3pt]{d2}{a1}
\ncline{d2}{f1}
\ncline{d2}{c1}
\ncline[linestyle=dotted,linewidth=1.3pt]{a1}{c1}
\ncline{f1}{c1}
\ncline{c1}{g1}
\rput[b](0,-35){\small $P^7$}
\end{pspicture}
\hfill\ 
\caption{Irreducible maps with large faces on the Projective Plane}
\label{mapsproj}
\end{figure}
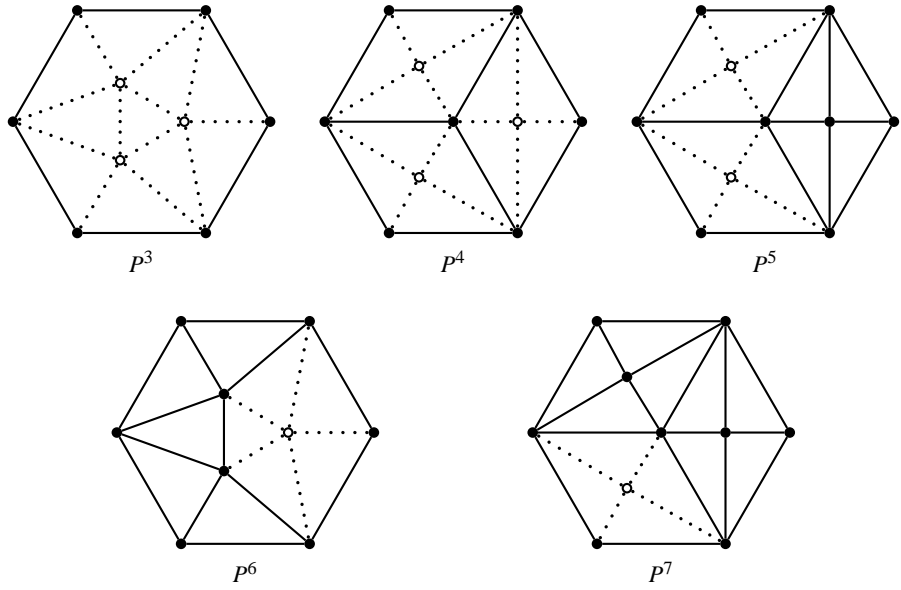

\begin{figure}[H]
  \small
  \centering
  \begin{minipage}[t]{0.3\textwidth}
    \input{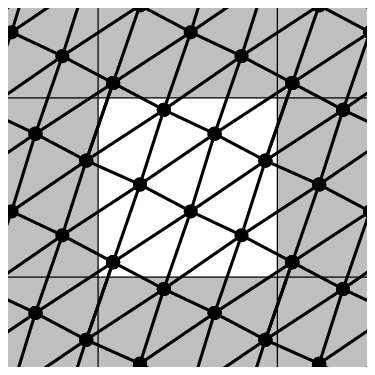}
  \end{minipage}%
  \hfill
  \begin{minipage}[t]{0.3\textwidth}
    \input{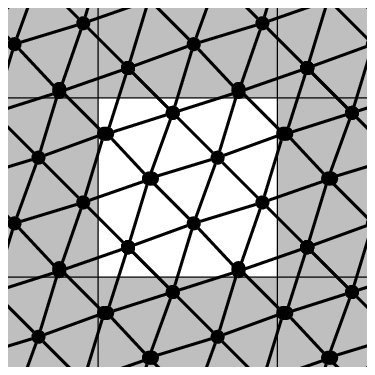}
  \end{minipage}%
  \hfill
  \begin{minipage}[t]{0.3\textwidth}
    \input{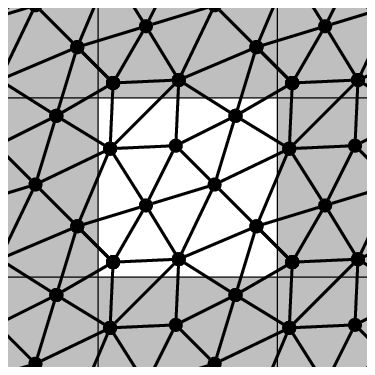}
  \end{minipage}%
  
  \bigskip
  
  \begin{minipage}[t]{0.3\textwidth}
    \input{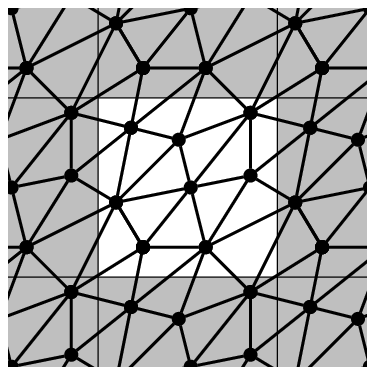}
  \end{minipage}%
  \hfill
  \begin{minipage}[t]{0.3\textwidth}
    \input{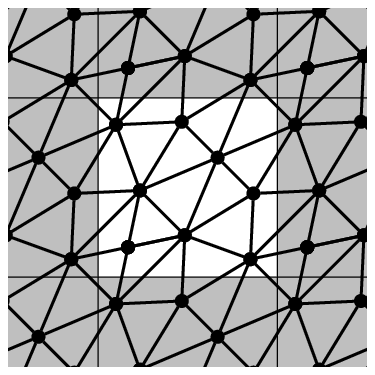}
  \end{minipage}%
  \hfill
  \begin{minipage}[t]{0.3\textwidth}
    \input{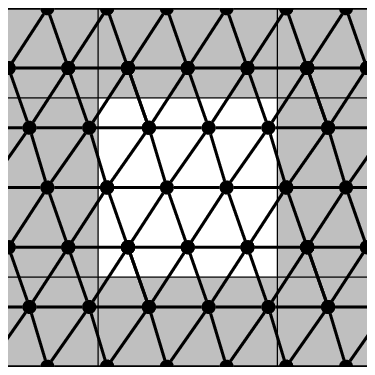}
  \end{minipage}%
  
  \bigskip
  
  \begin{minipage}[t]{0.3\textwidth}
    \input{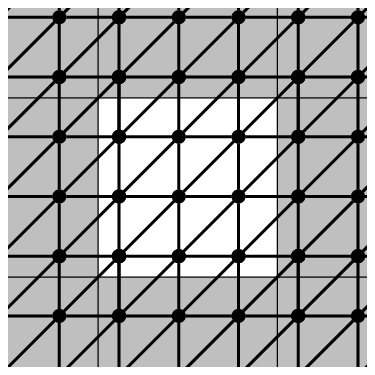}
  \end{minipage}%
  \hfill
  \begin{minipage}[t]{0.3\textwidth}
    \input{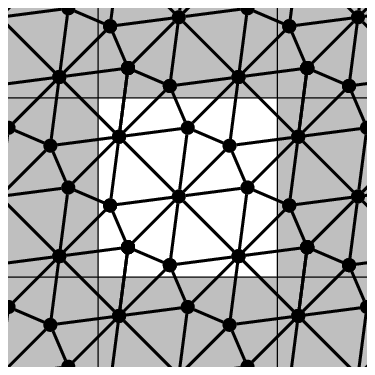}
  \end{minipage}%
  \hfill
  \begin{minipage}[t]{0.3\textwidth}
    \input{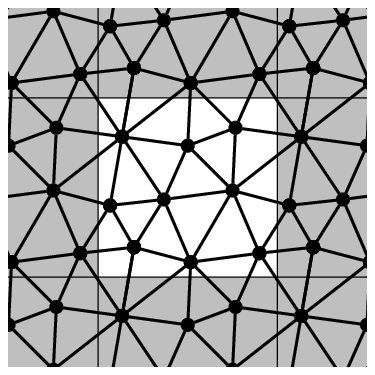}
  \end{minipage}%
  
  \bigskip
  
  \begin{minipage}[t]{0.3\textwidth}
    \input{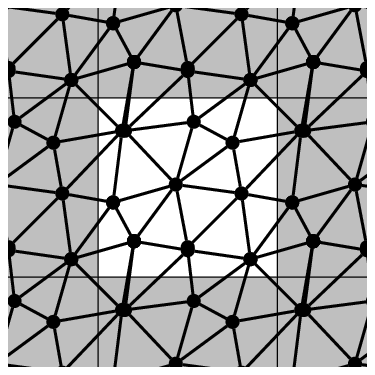}
  \end{minipage}%
  \hfill
  \begin{minipage}[t]{0.3\textwidth}
    \input{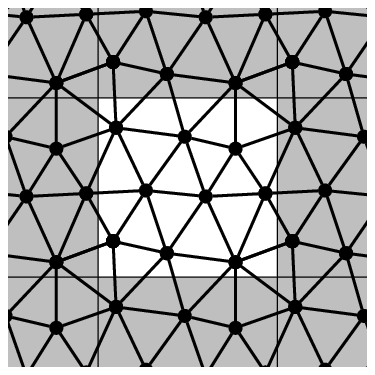}
  \end{minipage}%
  \hfill
  \begin{minipage}[t]{0.3\textwidth}
    \input{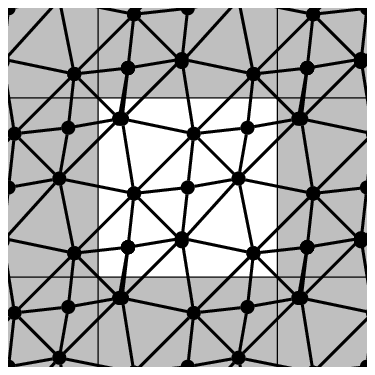}
  \end{minipage}%

  \bigskip
  
  \caption{Irreducible triangulations of the torus, $T^1$--$T^{12}$}
  \label{irrtorus1}
\end{figure}

\newpage

\begin{figure}[H]
  \small
  \centering
  \begin{minipage}[t]{0.3\textwidth}
    \input{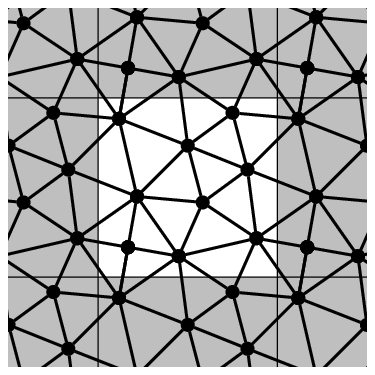}
  \end{minipage}%
  \hfill
  \begin{minipage}[t]{0.3\textwidth}
    \input{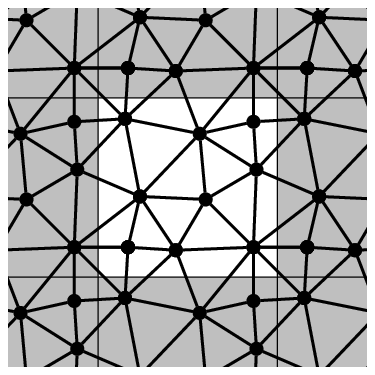}
  \end{minipage}%
  \hfill
  \begin{minipage}[t]{0.3\textwidth}
    \input{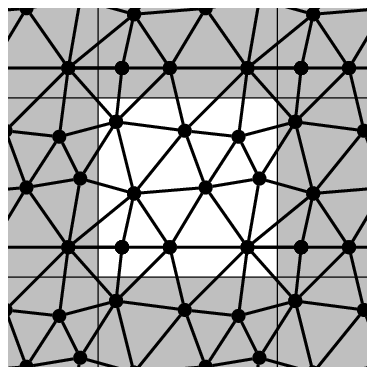}
  \end{minipage}%
  
  \bigskip
  
  \begin{minipage}[t]{0.3\textwidth}
    \input{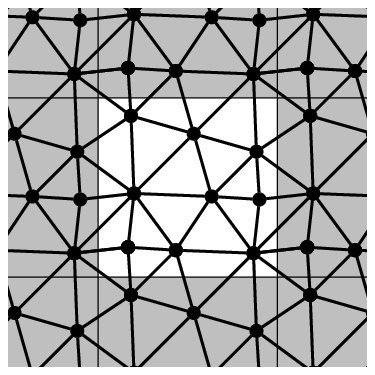}
  \end{minipage}%
  \hfill
  \begin{minipage}[t]{0.3\textwidth}
    \input{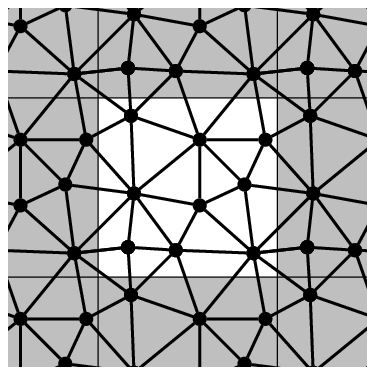}
  \end{minipage}%
  \hfill
  \begin{minipage}[t]{0.3\textwidth}
    \input{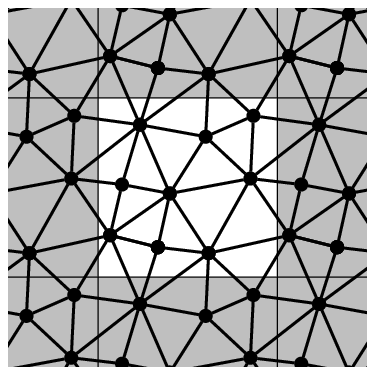}
  \end{minipage}%
  
  \bigskip
  
  \begin{minipage}[t]{0.3\textwidth}
    \input{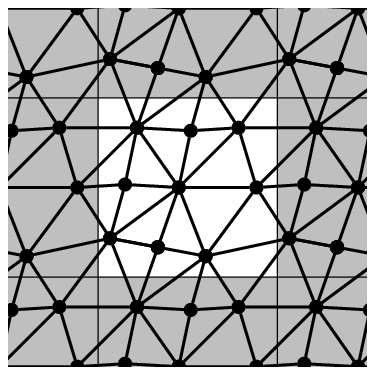}
  \end{minipage}%
  \hfill
  \begin{minipage}[t]{0.3\textwidth}
    \input{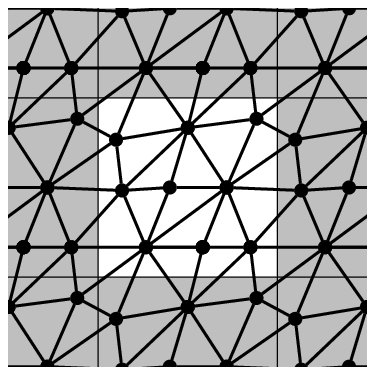}
  \end{minipage}%
  \hfill
  \begin{minipage}[t]{0.3\textwidth}
    \input{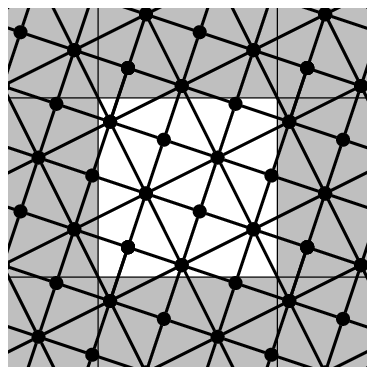}
  \end{minipage}%
  
  \bigskip
  
  \caption{Irreducible triangulations of the torus, $T^{13}$--$T^{21}$}
  \label{irrtorus2}
\end{figure}

\newpage

\begin{figure}[H]
  \small
  \centering
  \begin{minipage}[t]{0.3\textwidth}
    \input{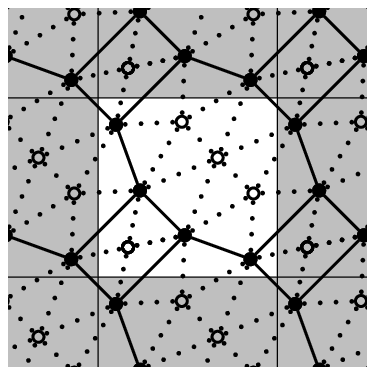}
  \end{minipage}%
  \hfill
  \begin{minipage}[t]{0.3\textwidth}
    \input{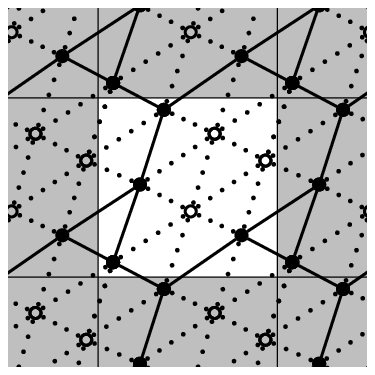}
  \end{minipage}%
  \hfill
  \begin{minipage}[t]{0.3\textwidth}
    \input{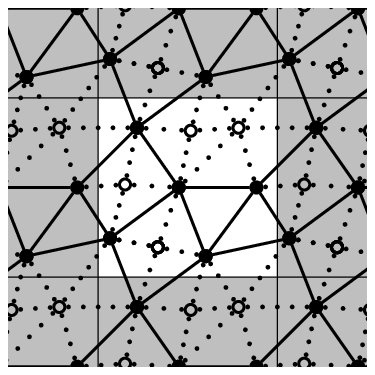}
  \end{minipage}%
  
  \bigskip
  
  \begin{minipage}[t]{0.3\textwidth}
    \input{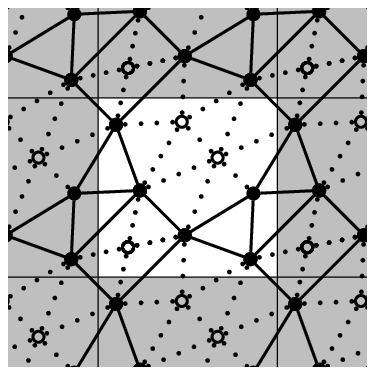}
  \end{minipage}%
  \hfill
  \begin{minipage}[t]{0.3\textwidth}
    \input{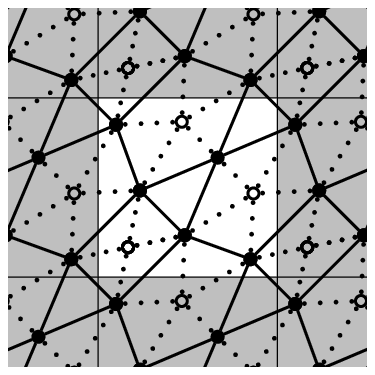}
  \end{minipage}%
  \hfill
  \begin{minipage}[t]{0.3\textwidth}
    \input{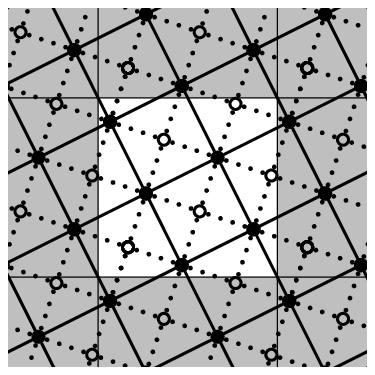}
  \end{minipage}%
  
  \bigskip
  
  \begin{minipage}[t]{0.3\textwidth}
    \input{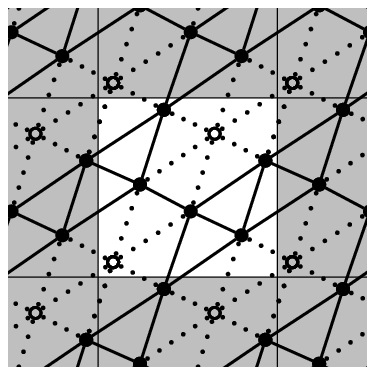}
  \end{minipage}%
  \hfill
  \begin{minipage}[t]{0.3\textwidth}
    \input{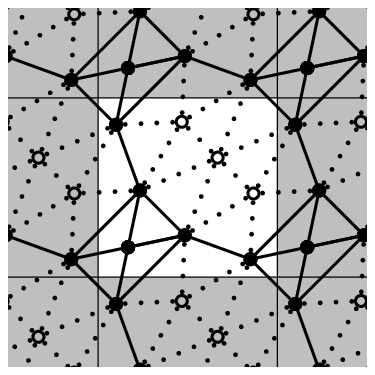}
  \end{minipage}%
  \hfill
  \begin{minipage}[t]{0.3\textwidth}
    \input{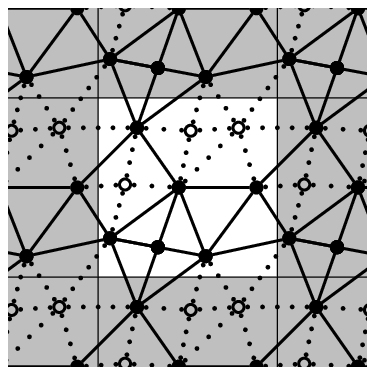}
  \end{minipage}%
  
  \bigskip
  
  \begin{minipage}[t]{0.3\textwidth}
    \input{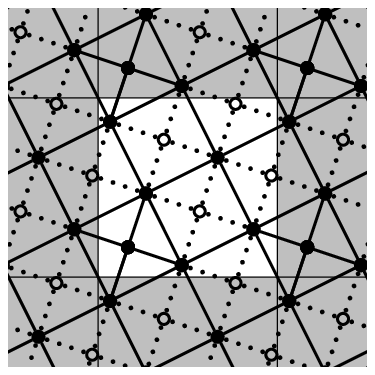}
  \end{minipage}%
  \hfill
  \begin{minipage}[t]{0.3\textwidth}
    \input{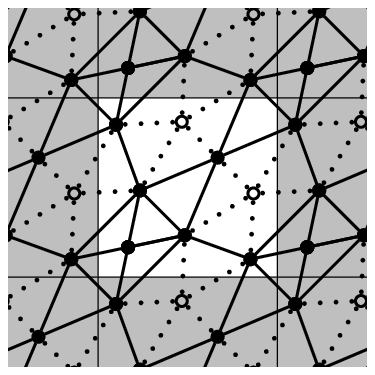}
  \end{minipage}%
  \hfill
  \begin{minipage}[t]{0.3\textwidth}
    \input{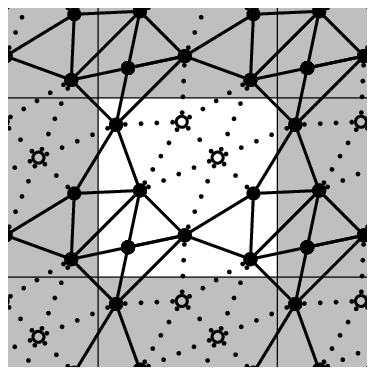}
  \end{minipage}%

  \bigskip
  
  \caption{Irreducible triangulations with large faces on the torus, $T^{22}$--$T^{33}$}
  \label{mapstorus1}
\end{figure}

\newpage

\begin{figure}[H]
  \begin{minipage}[t]{0.3\textwidth}
    \input{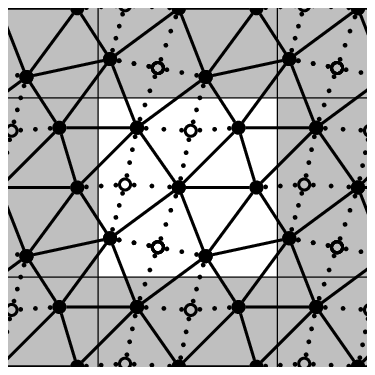}
  \end{minipage}%
  \hfill
  \begin{minipage}[t]{0.3\textwidth}
    \input{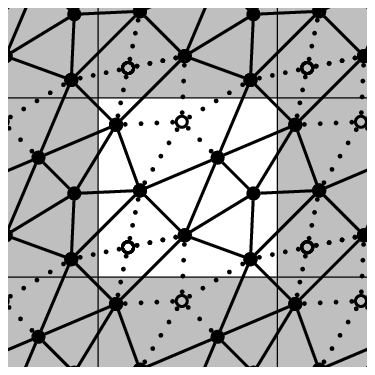}
  \end{minipage}%
  \hfill
  \begin{minipage}[t]{0.3\textwidth}
    \input{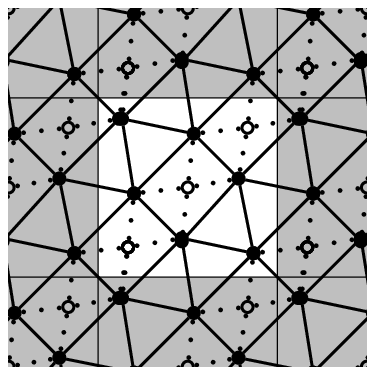}
  \end{minipage}%
  
  \bigskip
  
  \begin{minipage}[t]{0.3\textwidth}
    \input{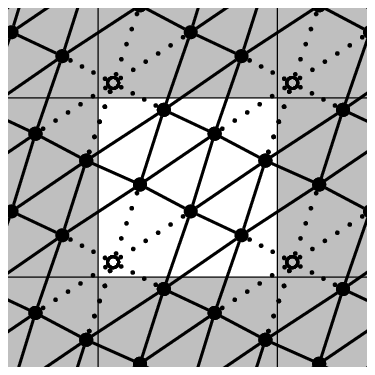}
  \end{minipage}%
  \hfill
  \begin{minipage}[t]{0.3\textwidth}
    \input{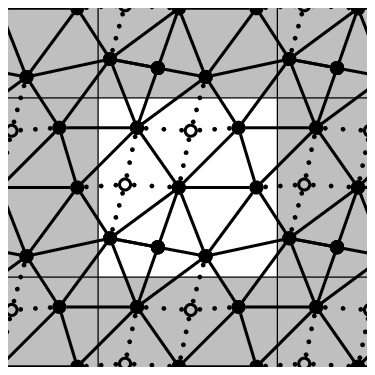}
  \end{minipage}%
  \hfill
  \begin{minipage}[t]{0.3\textwidth}
    \input{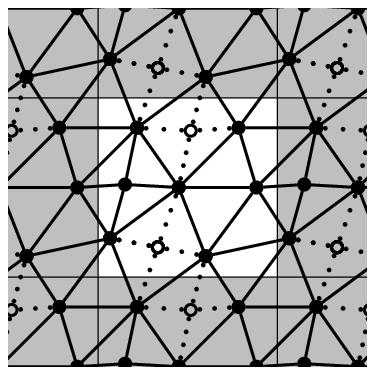}
  \end{minipage}%
  
  \bigskip
  
  \begin{minipage}[t]{0.3\textwidth}
    \input{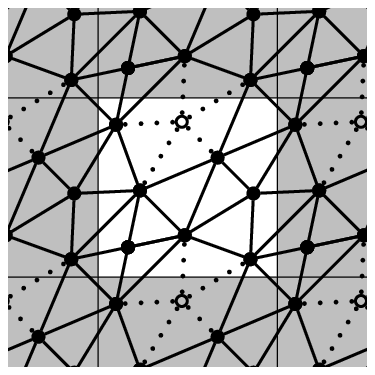}
  \end{minipage}%
  \hfill
  \begin{minipage}[t]{0.3\textwidth}
    \input{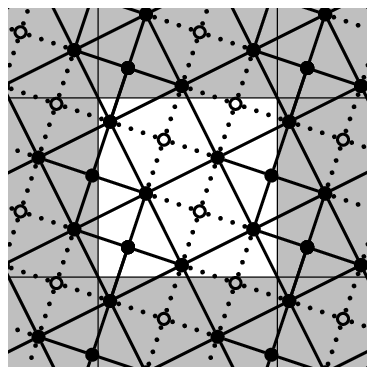}
  \end{minipage}%
  \hfill
  \begin{minipage}[t]{0.3\textwidth}
    \input{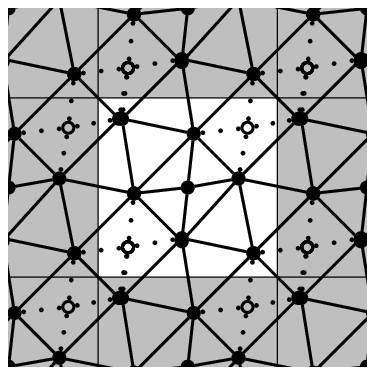}
  \end{minipage}%
  
  \bigskip
  
  \begin{minipage}[t]{0.3\textwidth}
    \input{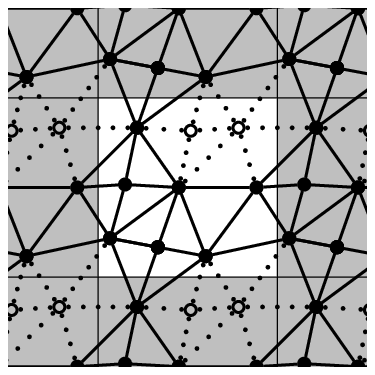}
  \end{minipage}%
  \hfill
  \begin{minipage}[t]{0.3\textwidth}
    \input{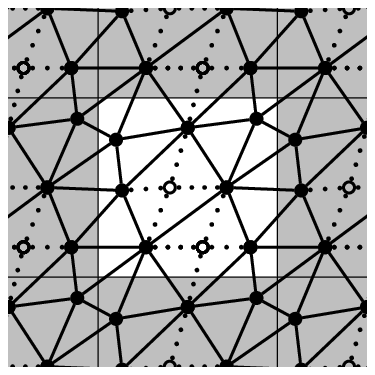}
  \end{minipage}%
  \hfill
  \begin{minipage}[t]{0.3\textwidth}
    \input{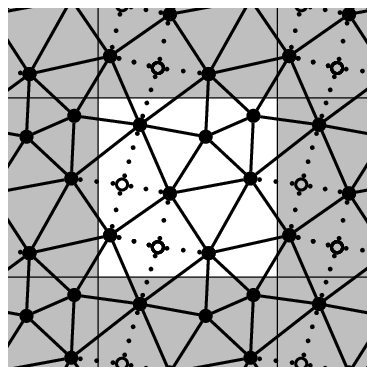}
  \end{minipage}%

  \bigskip
  
  \caption{Irreducible triangulations with large faces on the torus, $T^{34}$--$T^{45}$}
  \label{mapstorus2}
\end{figure}

\newpage

\begin{figure}[H]
  \begin{minipage}[t]{0.3\textwidth}
    \input{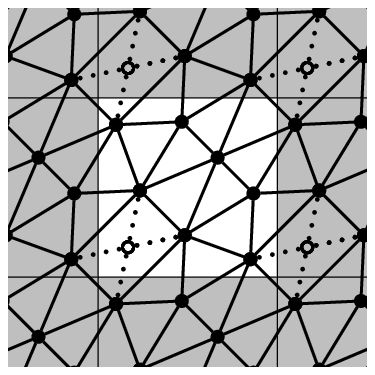}
  \end{minipage}%
  \hfill
  \begin{minipage}[t]{0.3\textwidth}
    \input{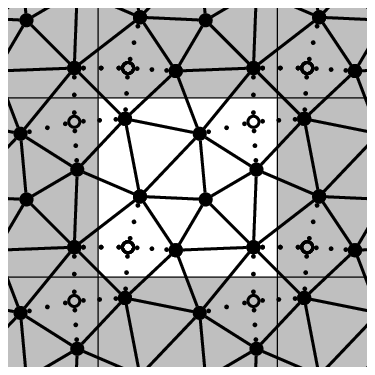}
  \end{minipage}%
  \hfill
  \begin{minipage}[t]{0.3\textwidth}
    \input{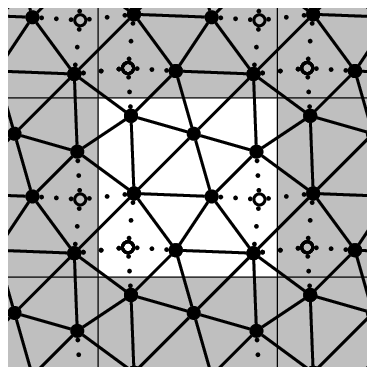}
  \end{minipage}%
  
  \bigskip
  
  \begin{minipage}[t]{0.3\textwidth}
    \input{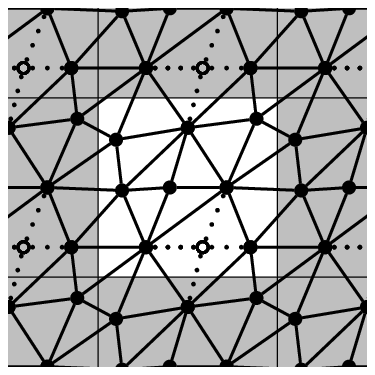}
  \end{minipage}%
  \hfill
  \begin{minipage}[t]{0.3\textwidth}
    \input{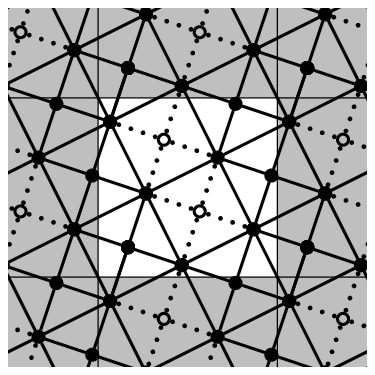}
  \end{minipage}%
  \hfill
  \begin{minipage}[t]{0.3\textwidth}
    \input{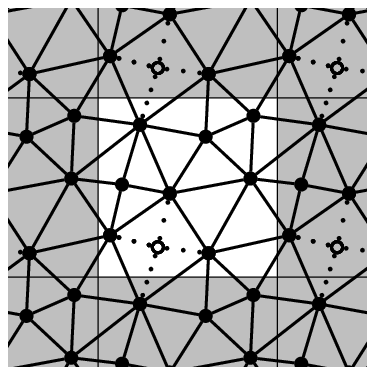}
  \end{minipage}%
  
  \bigskip
  
  \begin{minipage}[t]{0.3\textwidth}
    \input{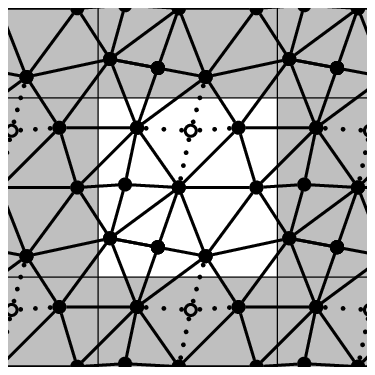}
  \end{minipage}%
  \hfill
  \begin{minipage}[t]{0.3\textwidth}
    \input{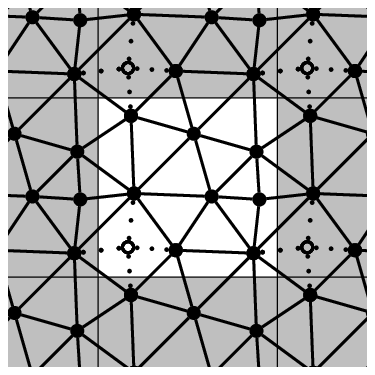}
  \end{minipage}%
  \hfill
  \begin{minipage}[t]{0.3\textwidth}
    \input{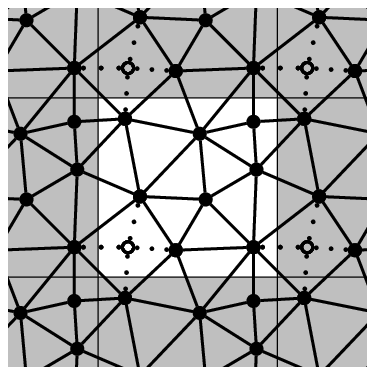}
  \end{minipage}%
  
  \bigskip
  
  \begin{minipage}[t]{0.3\textwidth}
    \input{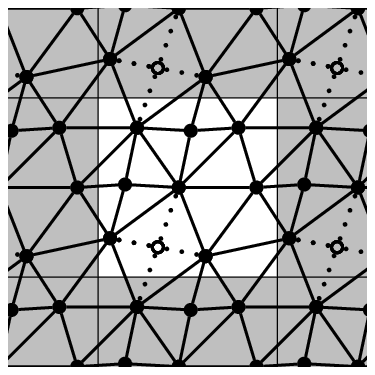}
  \end{minipage}%
  \hfill
  \begin{minipage}[t]{0.3\textwidth}
    \input{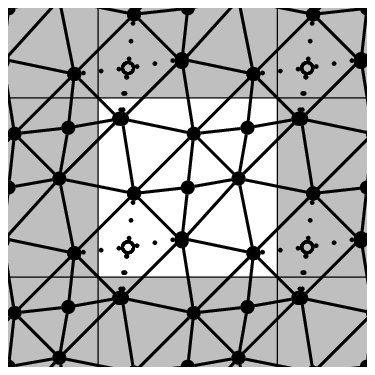}
  \end{minipage}%
  \hfill
  \begin{minipage}[t]{0.3\textwidth}
    \input{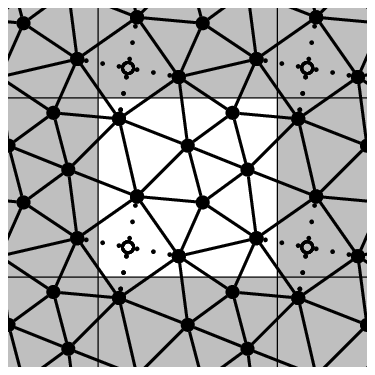}
  \end{minipage}%

  \bigskip
  
  \caption{Irreducible triangulations with large faces on the torus, $T^{46}$--$T^{57}$}
  \label{mapstorus3}
\end{figure}

\newpage

\begin{figure}[H]
  \begin{minipage}[t]{0.3\textwidth}
    \input{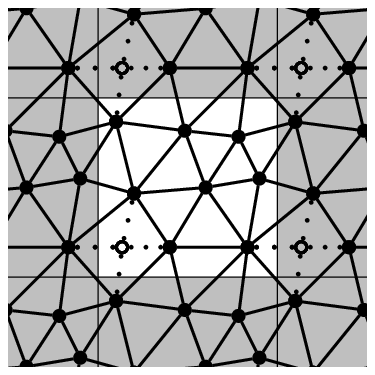}
  \end{minipage}%
  \hfill
  \begin{minipage}[t]{0.3\textwidth}
    \input{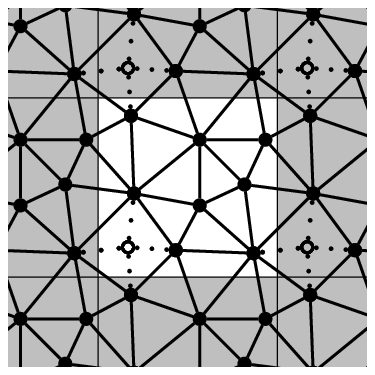}
  \end{minipage}%
  \hfill
  \begin{minipage}[t]{0.3\textwidth}
    \input{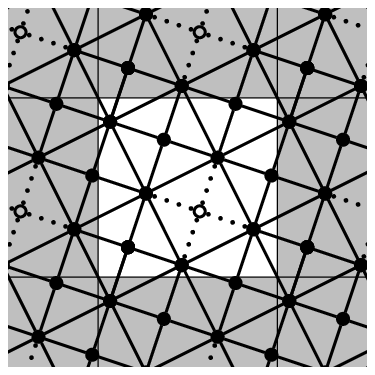}
  \end{minipage}%
  
  \bigskip
  
  \begin{minipage}[t]{0.3\textwidth}
    \input{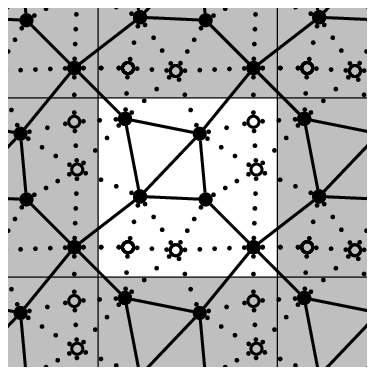}
  \end{minipage}%
  \hfill
  \begin{minipage}[t]{0.3\textwidth}
    \input{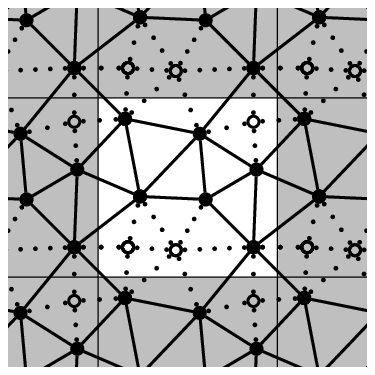}
  \end{minipage}%
  \hfill
  \begin{minipage}[t]{0.3\textwidth}
    \input{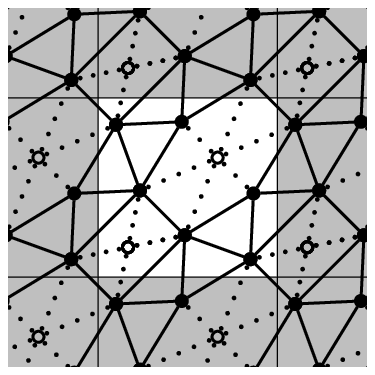}
  \end{minipage}%
  
  \bigskip
  
  \begin{minipage}[t]{0.3\textwidth}
    \input{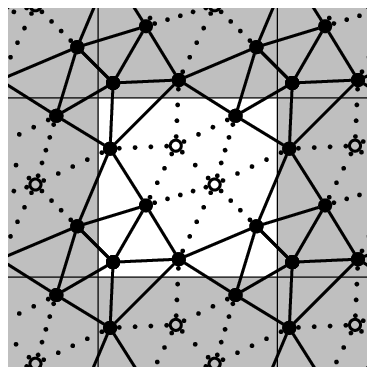}
  \end{minipage}%
  \hfill
  \begin{minipage}[t]{0.3\textwidth}
    \input{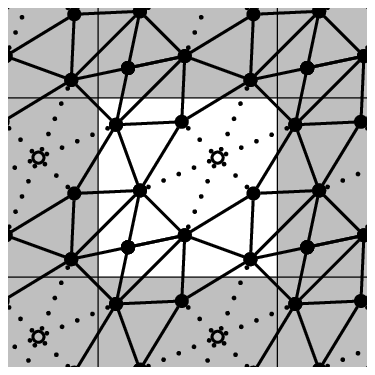}
  \end{minipage}%
  \hfill
  \begin{minipage}[t]{0.3\textwidth}
    \input{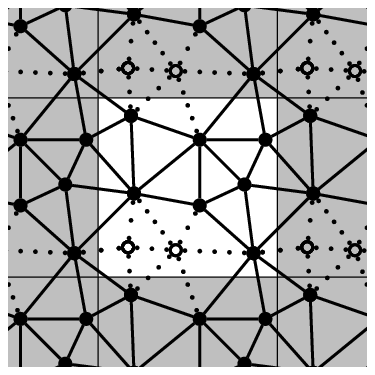}
  \end{minipage}%
  
  \bigskip
  
  \begin{minipage}[t]{0.3\textwidth}
    \input{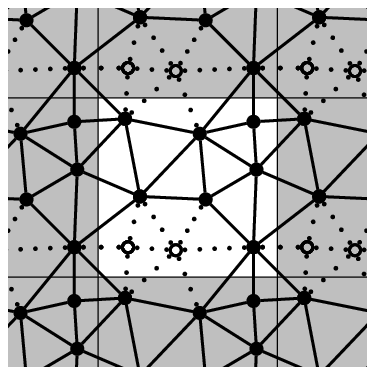}
  \end{minipage}%
  \hfill
  \begin{minipage}[t]{0.3\textwidth}
    \input{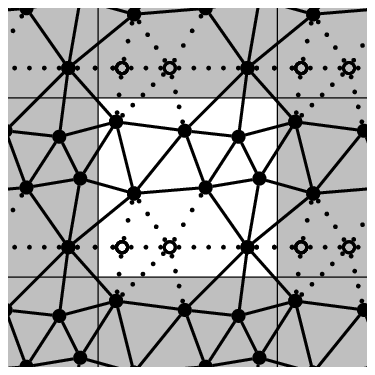}
  \end{minipage}%

  \bigskip
  
  \caption{Irreducible triangulations with large faces on the torus, $T^{58}$--$T^{68}$}
  \label{mapstorus4}
\end{figure}

\begin{landscape}
\begin{table}
\caption{Counts by number of vertices for maps on the projective plane}
\begin{tabular}{r|r r r|r r r|r r r}
   &\multicolumn{3}{c|}{Irreducible maps}&\multicolumn{3}{c|}{Face irreducible maps}&\multicolumn{3}{c}{Maps}\\
   &     & closed & open   &            &               &             &            &               &             \\
vertices & tri. & $2$-cell & $2$-cell & triangular & closed $2$-cell & open $2$-cell & polyhedral & closed $2$-cell & open $2$-cell \\
\hline
3  &  0 &  0 &  1 & 0            & 0             & 1             & 0             & 0             & 1  \\
4  &  0 &  1 &  1 & 0            & 1             & 2             & 0             & 1             & 6  \\
5  &  0 &  2 &  2 & 0            & 2             & 4             & 0             & 9             & 65 \\
6  &  1 &  2 &  2 & 1            & 5             & 11            & 1             & 188           & 1128 \\
7  &  1 &  1 &  1 & 3            & 18            & 42            & 10            & 4850          & 27041 \\
8  &  0 &  0 &  0 & 16           & 90            & 204           & 247           & 141255        & 741730 \\
9  &  0 &  0 &  0 & 134          & 566           & 1185          & 8576          & 4138394       & 21858099 \\
10 &  0 &  0 &  0 & 1210         & 4059          & 7768          & 263539        & 119621784     & 663735117 \\
11 &  0 &  0 &  0 & 11719        & 31773         & 55832         & 7290012       & 3413905527    & 20513612018 \\
12 &  0 &  0 &  0 & 114478       & 261912        & 427171        & 185392686     & 96571642059   & 640173121863 \\
13 &  0 &  0 &  0 & 1108826      & 2222281       & 3398145       & 4448447624    & 2716506356624 & -  \\
14 &  0 &  0 &  0 & 10606795     & 19187942      & 27792593      & 102469338545  & -             & -  \\
15 &  0 &  0 &  0 & 100352404    & 167528480     & 231869302     & 2292650424419 & -             & -  \\
16 &  0 &  0 &  0 & 940956644    & 1473754923    & 1963615254    & -             & -             & -  \\
17 &  0 &  0 &  0 & 8762227629   & 13035546948   & 16822695958   & -             & -             & -  \\
18 &  0 &  0 &  0 & 81168427279  & 115780306285  & 145453671164  & -             & -             & -  \\
19 &  0 &  0 &  0 & 748953936818 & 1031742846949 & 1267008314524 & -             & -             & -  \\
\end{tabular}
\label{MapcountsProj}
\end{table}
\end{landscape}

\begin{landscape}
\begin{table}
\caption{Counts by number of vertices for maps on the torus}
\begin{tabular}{r|r r r|r r r|r r r}
   &\multicolumn{3}{c|}{Irreducible maps}&\multicolumn{3}{c|}{Face irreducible maps}&\multicolumn{3}{c}{Maps}\\
   &     & closed & open   &            &               &             &            &               &             \\
vertices & tri. & $2$-cell & $2$-cell & triangular & closed $2$-cell & open $2$-cell & polyhedral & closed $2$-cell & open $2$-cell \\
\hline
4  & 0  & 0  & 2  & 0            & 0            & 2             & 0             & 0             & 3 \\
5  & 0  & 2  & 7  & 0            & 2            & 8             & 0             & 3             & 70 \\
6  & 0  & 7  & 11 & 0            & 8            & 25            & 0             & 205           & 2656 \\
7  & 1  & 12 & 16 & 1            & 39           & 124           & 1             & 15958         & 126466 \\
8  & 4  & 15 & 15 & 7            & 290          & 789           & 33            & 1014018       & 6070817 \\
9  & 15 & 16 & 16 & 112          & 2584         & 5976          & 4713          & 52587939      & 280232378 \\
10 & 1  & 1  & 1 & 2109         & 25202        & 49677         & 442429        & 2376996732    & 12389481487 \\
11 & 0  & 0  & 0  & 37867        & 258518       & 443088        & 28635972      & 97845502685   & 527699182180 \\
12 & 0  & 0  & 0  & 605496       & 2719039      & 4145672       & 1417423218    & 3770598166962 & - \\
13 & 0  & 0  & 0  & 8778329      & 28922902     & 40158242      & 58321972887   & -             & - \\
14 & 0  & 0  & 0  & 117839254    & 308507103    & 398230859     & 2102831216406 & -             & - \\
15 & 0  & 0  & 0  & 1491505713   & 3283573624   & 4008448713    & -             & -             & - \\
16 & 0  & 0  & 0  & 18035839188  & 34773632099  & 40687327452   & -             & -             & - \\
17 & 0  & 0  & 0  & 210391127053 & 365879925813 & 414537434014  & -             & -             & - \\
\end{tabular}
\label{MapcountsTorus}
\end{table}
\end{landscape}

\newpage


\bibliographystyle{amsplain}

\begin{thebibliography}{10}

\bibitem{MR1021367}
D.~W. Barnette and Allan~L. Edelson, \emph{All {$2$}-manifolds have finitely
  many minimal triangulations}, Israel J. Math. \textbf{67} (1989), no.~1,
  123--128. \MR{MR1021367 (91e:57006)}

\bibitem{MR84f:57009}
David Barnette, \emph{Generating the triangulations of the projective plane},
  J. Combin. Theory Ser. B \textbf{33} (1982), no.~3, 222--230. \MR{84f:57009}

\bibitem{plantri}
Gunnar Brinkmann and Brendan~D. McKay, \emph{plantri, program for generation of
  certain types of planar graph}, \url{http://cs.anu.edu.au/~bdm/plantri/}.

\bibitem{MR2357364}
\bysame, \emph{Fast generation of planar graphs}, MATCH Commun. Math. Comput.
  Chem. \textbf{58} (2007), no.~2, 323--357. \MR{2357364 (2008g:05053)}

\bibitem{plantripaper}
\bysame, \emph{Fast generation of planar graphs (expanded version)}, 2008,
  \url{http://cs.anu.edu.au/~bdm/papers/plantri-full.pdf}.

\bibitem{MR914777}
Serge Lawrencenko, \emph{Irreducible triangulations of a torus}, Ukrain. Geom.
  Sb. (1987), no.~30, 52--62, ii. \MR{MR914777 (89c:57002)}

\bibitem{MR98h:05067}
Serge Lawrencenko and Seiya Negami, \emph{Irreducible triangulations of the
  {K}lein bottle}, J. Combin. Theory Ser. B \textbf{70} (1997), no.~2,
  265--291. \MR{98h:05067}

\bibitem{MR1606516}
Brendan~D. McKay, \emph{Isomorph-free exhaustive generation}, J. Algorithms
  \textbf{26} (1998), no.~2, 306--324. \MR{MR1606516 (98k:68132)}

\bibitem{MR1844449}
Bojan Mohar and Carsten Thomassen, \emph{Graphs on surfaces}, Johns Hopkins
  Studies in the Mathematical Sciences, Johns Hopkins University Press,
  Baltimore, MD, 2001. \MR{MR1844449 (2002e:05050)}

\bibitem{MR961150}
Neil Robertson and P.~D. Seymour, \emph{Graph minors. {VII}. {D}isjoint paths
  on a surface}, J. Combin. Theory Ser. B \textbf{45} (1988), no.~2, 212--254.
  \MR{961150 (89m:05072)}

\bibitem{StRa}
Ernst Steinitz and Hans Rademacher, \emph{Vorlesungen \"uber die {T}heorie der
  {P}olyeder}, Springer, Berlin, 1934.

\bibitem{gentriang}
Thom Sulanke, \emph{Generating irreducible triangulations of surfaces}, arXiv:
  math.CO/0606687.

\bibitem{surftri}
\bysame, \emph{Source for surftri and lists of irreducible triangulations},
  \url{http://hep.physics.indiana.edu/~tsulanke/graphs/surftri/}.

\bibitem{math.CO/0407008}
\bysame, \emph{Note on the irreducible triangulations of the {K}lein bottle},
  J. Combin. Theory Ser. B \textbf{96} (2006), no.~6, 964--972.

\end{thebibliography}

\providecommand{\bysame}{\leavevmode\hbox to3em{\hrulefill}\thinspace}
\providecommand{\MR}{\relax\ifhmode\unskip\space\fi MR }
\providecommand{\MRhref}[2]{%
  \href{http://www.ams.org/mathscinet-getitem?mr=#1}{#2}
}
\providecommand{\href}[2]{#2}

\end{document}